\documentclass{amsart}
\usepackage{amsmath}
\usepackage{hyperref}
\usepackage{pgf,tikz}
\usepackage{placeins}
\usepackage[bottom=2cm,top=3cm,left=3cm,right=2cm]{geometry}
\usepackage{ytableau}
\usepackage{bbold}
\usepackage{enumitem}

\newtheorem{theorem}{Theorem}[section]
\newtheorem{lemma}[theorem]{Lemma}
\newtheorem{proposition}[theorem]{Proposition}
\newtheorem{corollary}[theorem]{Corollary}

\newtheorem{main lemma}[theorem]{Main Lemma}

\theoremstyle{definition}
\newtheorem{definition}[theorem]{Definition}
\newtheorem{example}[theorem]{Example}
\newtheorem{remark}[theorem]{Remark}

\DeclareMathOperator{\spn}{\mathrm{span}}

\begin{document}
\title[Cocharacters of $UT_n(E)$]{Cocharacters of $UT_n(E)$}
\author{Lucio Centrone}\address{Dipartimento di Matematica, Universit\`a degli Studi di Bari Aldo Moro, Via Edoardo Orabona, 4, 70125 Bari, Italy}
\email{lucio.centrone@uniba.it, centrone@unicamp.br}
\author{Vesselin Drensky}\address{Institute of Mathematics and Informatics, Bulgarian Academy of Sciences, 1113 Sofia, Bulgaria}\email{drensky@math.bas.bg}
\author{Daniela Martinez Correa}
\address{IMECC, Universidade Estadual de Campinas, Rua S\'ergio Buarque de Holanda, 651 Cidade Universit\'aria ``Zeferino Vaz'' Distr. Bar\~ao Geraldo Campinas, S\~ao Paulo, Brasil, CEP 13083-859}
\email{d190688@dac.unicamp.br}

\thanks{D.M. Correa was partially supported by CAPES-Brazil. Financial Code 001.}
\begin{abstract}
Let $F$ be a field of characteristic $0$ and let $E$ be the infinite dimensional Grassmann algebra over $F$.
In the first part of this paper we give an algorithm  calculating the generating function of the cocharacter sequence
of the $n\times n$ upper triangular matrix algebra $UT_n(E)$ with entries in $E$, lying in a strip of a fixed size.
In the second part we compute the double Hilbert series  $H(E;\mathrm{T}_k,\mathrm{Y}_l)$ of $E$, then we define the $(k,l)$-multiplicity series of any PI-algebra.
As an application, we derive from $H(E;\mathrm{T}_k,\mathrm{Y}_l)$ an easy algorithm determining the $(k,l)$-multiplicity series of $UT_n(E)$.
\end{abstract}

\maketitle

\section{Introduction}
We fix a field $F$ of characteristic $0$ and any algebra over $F$ is considered associative with unit.
Let $X=\{x_1,x_2,\ldots\}$ be a countable set of indeterminates. We denote by $F\langle X\rangle$ the free algebra freely generated by $X$ over $F$.
Let $A$ be an algebra over $F$ satisfying a polynomial identity, i.e., a PI-algebra.
It is well known that its set of polynomial identities $T(A)$ is a $T$-ideal of $F\langle X\rangle$, i.e., an ideal that is invariant under all endomorphisms of $F\langle X\rangle$.\\
\indent
Since $F$ is a field of characteristic 0, all the polynomial identities follow from the multilinear ones.
A famous theorem by Kemer \cite{Ke} says that if $A$ is a PI-algebra, then its $T$-ideal is finitely generated,
but it is important to  recall that the complete set of finite generators of $T$-ideals is well known  only for few algebras.\\
\indent
By a result of Regev \cite{R1}, it seems to be more efficient to study the set of multilinear polynomials which (in a certain sense) are not polynomial identities for a given algebra.
More precisely, if $P_n$ is the vector space of multilinear polynomials in the variables $\{x_1,\ldots,x_n\}$, we study the factor space $P_n(A):=P_n/(P_n\cap T(A))$ for each $n$.
We recall that $P_n$ is also a left $S_n$-module under the canonical left action of the symmetric group $S_n$.
Since $P_n(A)$ inherits the $S_n$-action on $P_n$, it affords an $S_n$-character $\chi_n(A)$ called the \textit{$n$-th cocharacter of $A$}.
The sequence $(\chi_n(A))_{n\in\mathbb{N}}$ is called the \textit{sequence of cocharacters of $A$}.
We also observe that $P_n(A)$ is a finite dimensional vector space which dimension is called the \textit{$n$-th codimension} of $A$ (or in symbol $c_n(A)$)
and the sequence $(c_n(A))_{n\in\mathbb{N}}$ is called the \textit{sequence of codimensions of $A$}.\\
\indent
In \cite{GZ2}, \cite{GZ1}, see also \cite{GZ}, Giambruno and Zaicev proved that there always exists the limit
\[
\text{\rm exp}(A)=\lim_{n\rightarrow\infty}\sqrt[n]{c_n(A)}
\]
and it is a nonnegative integer called the \textit{PI-exponent of $A$}.
If we use the language of varieties, we say that the variety generated by the algebra $A$ is the class
\[
\mathcal{V}=\mathcal{V}(A)=\{\text{\rm $B$ associative algebra}\mid T(A)\subseteq T(B)\}.
\]
We say that the variety of algebra $\mathcal{V}$ is \textit{minimal} with respect to its exponent
if and only if for any proper subvariety $\mathcal{U}$ of $\mathcal{V}$
we have that $\text{\rm exp}(\mathcal{U})<\text{\rm exp}(\mathcal{V})$.
We say that a PI-algebra is \textit{minimal} if it generates a minimal variety.\\
\indent
If $S$ is any commutative ring with 1, we denote by $UT_n(S)$ the ring of upper triangular matrices with entries in $S$.
Let $E$ be the infinite dimensional Grassmann algebra over $F$. Drensky \cite{D3} proved that the $T$-ideals of the algebras $UT_n(F)$ and $UT_n(E)$
are examples of maximal $T$-ideals of a given exponent of the codimension sequences
(and the corresponding varieties of algebras are minimal varieties of this exponent).
Some years before Kemer's works, Genov in \cite{gen1} and \cite{gen2} Genov and Latyshev in \cite{Ly} proved
that every algebra belonging to $\mathcal{V}(UT_n(F))$ has a finite basis of its polynomial identities.
Latyshev in \cite{Ly2} and Popov in \cite{pop} generalized the previous result for PI-algebras satisfying the polynomial identity
\[
[x_1,x_2,x_3]\cdots[x_{3n-2},x_{3n-1},x_{3n}]
\]
which generates the $T$-ideal $T(UT_n(E))=T(E)^n$ of the algebra $UT_n(E)$.
For a long time, until Kemer developed his structure theory, the results of Genov, Latyshev and Popov
covered all known examples of classes of PI-algebras with the finite basis property.\\
\indent
The T-ideals $T(UT_n(F))$ and $T(UT_n(E))$ have another interesting property established by Volichenko and Zalesskii in \cite{VZ}.
Let the algebra $A$ satisfy a multilinear polynomial identity $f(x_1,\ldots,x_m)$ which generates an irreducible $S_m$-module
with character $\chi_{\lambda}$, where $\chi_\lambda$ is the irreducible $S_m$-character associated with the partition $\lambda$.
Then the Young diagram of $\lambda$ contains less than $n$ boxes below of the first row
if and only if $f(x_1,\ldots,x_m)$ does not hold for $UT_n(F)$. Similarly, the Young diagram of $\lambda$ contains less than $n$ boxes to the right of the first column
if and only if $f(x_1,\ldots,x_m)$ does not hold for $UT_n(E)$.
For the Grassmann algebra this means that the algebra $A$ satisfies a standard identity if and only if $T(A)$ is not contained in $T(E)$.
A proof can be found for example in the book by Giambruno and Zaicev \cite[Theorem 7.2.1]{GZ}.

Let
\[
\chi_n(A)=\sum_{\lambda\vdash n}m_{\lambda}(A)\chi_\lambda,\quad n\in\mathbb{N},
\]
be the cocharacter sequence of $A$.
Let us set $X_d:=\{x_1,\ldots,x_d\}$ and let us consider 
\[
F_d(A):=F\langle X_d\rangle/(F\langle X_d\rangle\cap T(A)).
\]
Moreover, if $\mathrm{T}=\{t_1,\ldots,t_d\}$ is a set of commutative variables,
then the Hilbert series $H(F_d(A);T_d)$ of $F_d(A)$ may be decomposed as
\[
H(F_d(A);\mathrm{T}_d)=\sum_{\lambda}m_\lambda(A)S_\lambda(\mathrm{T}_d),
\]
where $\lambda$ is a partition in no more than $d$ parts and $S_\lambda(\mathrm{T}_d)$ is the Schur function associated to $\lambda$ in the variables from $\mathrm{T}_d$.
We shall refer to $H(F_d(A);\mathrm{T}_d)$ as the {\it Hilbert series of $A$} and we shall write $H(A,\mathrm{T}_d)$ instead of $H(F_d(A);\mathrm{T}_d)$.
By a result of Berele and Drensky, (see \cite{B1} and \cite{D1}), the $m_\lambda(A)$'s are the same as in the cocharacter sequence of $A$.
Hence, in principle, the knowledge of the Hilbert series of $A$ will give us the multiplicities $m_\lambda(A)$ of the cocharacter sequence of $A$,
when $\lambda$ is a partition in no more than $d$ parts.
So if $A$ is finite dimensional, working with a sufficiently large set of variables will be enough to capture all the multiplicities. This is no longer true for infinite dimensional algebras. It is also important to recall that Belov proved in \cite{Belov} that the Hilbert series
of the relatively free algebra of a PI-algebra A in $d$ variables is a {\it rational} function.\\
\indent
The explicit form of the multiplicities in the cocharacter sequence of a PI-algebra is known for few cases. Among them are the infinite dimensional Grassmann algebra $E$ (Olsson and Regev \cite{OR}),
the $2\times 2$ matrix algebra $M_2(F)$ (Formanek \cite{for1}  and Drensky \cite{D4}), the algebra $UT_2(F)$ of $2\times 2$ upper triangular matrices
(Mishchenko et al \cite{Mi}, based on the approach of Berele and Regev \cite{BR3}, see also \cite{D2}),
the tensor square $E\otimes E$ of the Grassmann algebra (Popov \cite{pop1}, Carini and Di Vincenzo \cite{CAD}),
the algebra $UT_2(E)$ of $2\times 2$ upper triangular matrices with entries from the  Grassmann algebra $E$ (Centrone \cite{C}),
the algebra $UT_n(F)$ of $n \times n$ upper triangular matrices (Boumova and Drensky \cite{BD}),
the algebra $R_{p,q}(F)$ of upper block triangular $(p+2q)\times (p+2q)$ when $p$ and $q$ are small values (Drensky and Kostadinov \cite{DK}).\\
\indent
In \cite{DG1} Drensky and Genov define the \textit{multiplicity series} of a PI-algebra $A$,
that is the generating function of the cocharacter sequence of $A$ which corresponds to the multiplicities $m_\lambda(A)$ when $\lambda$ is a partition in no more than $d$ parts.
Then, coming back to upper triangular matrices and their central role in PI-theory,
in \cite{BD} Boumova and Drensky found an easy algorithm with input the multiplicity series of a symmetric function, and output the multiplicity series of its Young-derived.
Applying it, they found the explicit form of the multiplicity series of the Hilbert series of $UT_n(F)$.
Following this line of research, in the first part of the paper we work with $UT_n(E)$ and calculate its multiplicity series in $d$ variables.\\
\indent
Due to the fact that $E$ is infinite dimensional, we need more tools than the ones used by Boumova and Drensky in order to know all multiplicities of $UT_n(E)$.
Using the idea of Berele (see \cite{B3}), we work with \textit{double Hilbert series} instead of with Hilbert series of PI-algebras.
Due to the analogue of the result of Berele and Drensky for double Hilbert series, it suffices to study the decomposition of the double Hilbert series of $UT_n(E)$
in order to achieve the explicit form of the cocharacter sequence of $UT_n(E)$.
In the second part of the present paper, we generalize the definition of multiplicity series of a PI-algebra defining a \textit{$(k,l)$-multiplicity series}
which controls three sets of disjoint variables, where $(k,l)$ means that the partitions $\lambda=(\lambda_1,\ldots, \lambda_m)$ satisfy the condition $\lambda_{k+1}\leq l$.
In other words, their young diagrams $D_{\lambda}$ are in a hook of height $k$ of the arm and wide $l$ of the leg.
By a result of Amitsur and Regev \cite{AR} all nonzero multiplicities $m_{\lambda}(A)$ for a PI-algebra $A$ are concentrated for Young diagrams in a sufficiently large hook.
Hence the information about the multiplicities of $A$ is contained in the related with the hook $(k,l)$-multiplicity series.\\
\indent
Then we compute the double Hilbert series of $E$ and, as a consequence, we build up an algorithm with output the $(k,l)$-multiplicity series of $UT_n(E)$.
In the spirit of \cite{C} we compute the $(2,3)$-multiplicity series of $UT_2(E)$, which contains all multiplicities of the cocharacter sequence of $UT_2(E)$
and finally we compute the $(1,1)$-multiplicity series of $UT_3(E)$.


\section{Preliminaries}
\subsection{Symmetric functions}
We fix a positive integer $d$ and consider the algebra
\[
\mathbb{C}[[\mathrm{T}_d]]=\mathbb{C}[[t_1,\ldots, t_d]]
\]
of formal power series in $d$ commutative variables. Let ${\mathbb{C}[[\mathrm{T}_d]]}^{S_d}\subseteq {\mathbb{C}[[\mathrm{T}_d]]}$ be the subalgebra of symmetric functions.
Every symmetric function $g(\mathrm{T}_d)$ can be represented in the form
\[
g(\mathrm{T}_d)=\sum\limits_{\lambda} m_{\lambda}S_{\lambda}(\mathrm{T}_d), m_{\lambda}\in\mathbb{C},\hspace{0.1cm}\lambda=(\lambda_1,\ldots, \lambda_d),
\]
where $S_{\lambda}(\mathrm{T}_d)$ is the Schur function related to the partition $\lambda$  which has at most $d$ parts.
For details on the theory of Schur functions see \cite{M}.\\
\indent
There are several ways to define Schur functions. The most convenient for our purpose is to define them as fractions of Vandermonde-type determinants:
\[
S_{\lambda}(\mathrm{T}_d)= \dfrac{V(\lambda+\delta,\mathrm{T}_d)}{V(\delta,\mathrm{T}_d)},
\]
where $\delta=(d-1,\ldots,2,1)$ and for $\mu=(\mu_1,\ldots, \mu_d)$
\vspace{0.2cm}
\[
V(\mu, \mathrm{T}_d)= \left| \begin{matrix}
t_1^{\mu_1} & t_2^{\mu_1} & \ldots & t_{m-1}^{\mu_1} & t_m^{\mu_1}\\
t_1^{\mu_2} & t_2^{\mu_2} & \ldots & t_{m-1}^{\mu_2} & t_m^{\mu_2}\\
\vdots &  \vdots  &\ddots & \vdots & \vdots\\
t_1^{\mu_{m-1}} & t_2^{\mu_{m-1}} & \ldots & t_{m-1}^{\mu_{m-1}} & t_m^{\mu_{m-1}}\\
t_1^{\mu_m} & t_2^{\mu_m} & \ldots & t_{m-1}^{\mu_m} & t_m^{\mu_m}
\end{matrix}\right|.
\]
Let $\lambda=(\lambda_1,\ldots,\lambda_d)$ be a partition of a natural number.
The Young diagram  $D_{\lambda}$ associated to $\lambda$ is the subset of $\mathbb{Z}\times\mathbb{Z}$ defined as
$D_{\lambda}=\{(i,j)\mid i=1,\ldots, d,\hspace{0.05cm} j=1,\ldots, \lambda_i\}$. Graphically we draw the diagrams replacing the
knots by square boxes, adopting the convention, as with matrices, that the first coordinate $i$ (the row index) increases as one goes downwards,
and the second coordinate $j$ (the column index) increases as one goes from left to right.
The first boxes from the left of each row are one above another and the $i$-th row contains $\lambda_i$ boxes.
We denote by $\lambda'_j$ the length of the $j$-th column of $D_{\lambda}$. The partition $\lambda' = (\lambda'_1 ,\ldots, \lambda'_m)$
and its diagram $D_{\lambda'}$ are called conjugate respectively to $\lambda$ and $D_{\lambda}$.\\
\indent
For the partition $\lambda$, we define a $\lambda$-tableau $T_{\lambda}$ of content $\alpha=\alpha(T_{\lambda})=(\alpha_1,\ldots,\alpha_d)$
if each integer $i=1,\ldots, d$ appears in the tableau exactly $\alpha_i$ times.
Recall that the $\lambda$-tableau $T_{\lambda}$ is semistandard if its entries do not decrease in rows reading from left to  right,
and increase strictly in columns reading from top to bottom.\\
\indent
Another definition of Schur functions is given in terms of semistandard Young tableaux:
\[
S_{\lambda}({\mathrm{T}_d})=\sum\mathrm{T}_d^{{\alpha}(T_{\lambda})},
\]
where the summations runs on all semistandard $\lambda$-tableaux.\\

We recall the definition of elementary symmetric polynomials. 
Given $0\leq m\leq d$, the $m$-th elementary symmetric polynomial in $d$ variables $t_1,\ldots, t_d$ is defined by
\[
e_m(\mathrm{T}_d)=\sum\limits_{1\leq i_1<\cdots<i_m\leq d}t_{i_1}\cdots t_{i_m}.
\]
These polynomials will be used  several times throughout this work.
It can be easily proved  that $e_{m}(\mathrm{T}_d)=S_{(1^m)}(\mathrm{T}_d)$.\\

For each partition $\mu$ of $n$ we denote by $M_{\mu}$ and $\chi_{\mu}$ the corresponding irreducible $S_n$-module and its character, respectively.
A very useful tool for the development of this work is the \textit{Young Rule}.
This rule describes in the language of Young diagrams the induced $S_{m+n}$-characters of the $S_{m}\times S_n$-characters
$\chi_{(m)}\otimes \chi_{\mu}$ and $\chi_{(1^m)}\otimes \chi_{\mu}$, $\mu\vdash n$,
associated to the modules $M_{(m)}\otimes M_{\mu}$ and $M_{(1^m)}\otimes M_{\mu}$. The induced $S_{m+n}$-characters of this modules are  denote by
We denote them by
$\chi_{(m)}\widehat{\otimes} \chi_{\mu}$ and $\chi_{(1^m)}\widehat{\otimes} \chi_{\mu}$, respectively.
In the special case $m=1$, the Young rule is equivalent to the \textit{Branching Rule} for the induced $S_{n+1}$-character of $\chi_{\mu}$, $\mu\vdash n$.
Translated in the language of the Schur functions, the Young rule can be stated as follows:\\

\noindent \textbf{Case 1.} Let $\mu=(\mu_1,\ldots, \mu_d)$ be a partition and $m\in\mathbb{N}$. Then
\[
S_{(m)}(\mathrm{T}_d)S_{\mu}(\mathrm{T}_d)=\sum\limits_{\lambda} S_{\lambda}(\mathrm{T}_d)
\]
where the summation runs over all partitions $\lambda$ such that
\[
\lambda_1+\cdots +\lambda_d=\mu_1+\cdots +\mu_d+m,
\]
\[
\lambda_1\geq \mu_1\geq \lambda_2\geq \mu_2\geq\cdots\geq\lambda_d\geq\mu_d.
\]
This means that the Young diagrams $D_{\lambda}$ are obtained from the diagram $D_{\mu}$ by adding $m$ boxes
so that two new boxes cannot lie in the same column of $D_{\lambda}$.\\

\noindent \textbf{Case 2.} Let $\mu=(\mu_1,\ldots, \mu_d)$ and $(1^m)$ be partitions with $m\leq d$. Then
\[
S_{(1^m)}(\mathrm{T}_d)S_{\mu}(\mathrm{T}_d)=\sum\limits_{\lambda} S_{\lambda}(\mathrm{T}_d)
\]
where the summation is over all partitions $\lambda$ such that
\[
\lambda_1+\cdots +\lambda_d= \mu_1+\cdots +\mu_d+ m,
\]
\[
\mu_i=\lambda_i+ \varepsilon_i, \varepsilon_i=0,1.
\]
In other words, the Young diagrams $D_\lambda$ are obtained from the diagram $D_{\mu}$ by adding $m$ boxes so that new boxes are not allowed to lie in the same row.\\

Let  $g(\mathrm{T}_d)=\sum\limits_{\lambda} m_{\lambda}S_{\lambda}(\mathrm{T}_d)$ be a symmetric function, then we define its \textit{multiplicity series} as
\[
M(g;\mathrm{T}_d)=\sum\limits_{\lambda} m_{\lambda}\mathrm{T}_d^{\lambda}=\sum\limits_{\lambda}m_{\lambda}t_1^{\lambda_1}\cdots  t_d^{\lambda_d}\in\mathbb{C}[[\mathrm{T}_d]].
\]
It is also convenient to consider the subalgebra $\mathbb{C}[[\mathrm{V}_d]]\subseteq\mathbb{C}[[\mathrm{T}_d]]$ of the formal power series in the new set of variables $\mathrm{V}_d=\{v_1,\ldots, v_d\}$, where
\[
v_1=t_1, v_2=t_1t_2 ,\ldots, v_d=t_1\cdots t_d.
\]
Then the multiplicity series $M(g;\mathrm{T}_d)$ can be written as
\[
M'(g;\mathrm{V}_d)=\sum\limits_{\lambda} m_{\lambda}v_1^{\lambda_1-\lambda_2}\cdots v_{d-1}^{\lambda_{d-1}-\lambda_{d}}v_d^{\lambda_d}\in\mathbb{C}[[\mathrm{V}_d]].
\]
We also call $M'(g;\mathrm{V}_d)$ the multiplicity series of $g$.
The advantage of the mapping $M': {\mathbb{C}[[\mathrm{T}_d]]}^{S_d}\longrightarrow\mathbb{C}[[\mathrm{V}_d]]$
defined by $g(\mathrm{T}_d)\rightarrow M'(g;\mathrm{V}_d)$ is that it is a bijection.
\\
\indent
The functions $g(\mathrm{T}_d)\in\mathbb{C}[[\mathrm{T}_d]]^{S_d}$ and $M(g;\mathrm{T}_d)$ are related by the following equality.

\begin{lemma}[Berele \cite{Berel2}]
If
\[
g(\mathrm{T}_d)\prod\limits_{i<j}(t_i-t_j)=\sum_{p_i\geq 0} b(p_1,\ldots,p_d) t_1^{p_1}\cdots t_d^{p_d},\quad b(p_1,\ldots, p_d)\in\mathbb{C},
\]
then
\[
M(g;\mathrm{T}_d)=\dfrac{1}{t_1^{d-1}t_2^{d-2}\cdots t_{d-1}}\sum\limits_{p_i>p_{i+1}}b(p_1,\ldots,p_d)t_1^{p_1}\cdots t_d^{p_d}
\]
where the summations in the latter equation runs on all $p=(p_1,\ldots,p_d)$ such that $p_1>p_2\cdots >p_d$.
\end{lemma}

In the general case, it is difficult to find $M(g;\mathrm{T}_d)$ even if we know $g(\mathrm{T}_d)$. But it is very easy to check whether the formal power series
\[
h(\mathrm{T}_d)=\sum\limits h(q_1,\ldots, q_d)t_1^{q_1}\cdots t_d^{q_d},  \quad q_1\geq \cdots \geq q_d,
\]
is equal to the multiplicity series $M(f;\mathrm{T}_d)$ of $f(\mathrm{T}_d)$ because $h(\mathrm{T}_d)=M(f;\mathrm{T}_d)$ if and only if
\[
f(\mathrm{T}_d)\prod\limits_{i<j}(t_i-t_j)=\sum\limits_{\sigma\in S_{d}}t_{\sigma(1)}^{d-1}t_{\sigma(2)}^{d-2}\cdots t_{\sigma(d-1)}h(t_{\sigma(1)},\ldots,t_{\sigma(d)}).
\]
The latter equation can be used to verify the computational results on multiplicities.
\subsection{PI-algebras}
Let $F$ be a field of characteristic $0$, let $A$ be an associative $F$-algebra with unity, and let $X=\{x_1,x_2,\ldots\}$ be a countable set of variables.
We denote by $F\langle X\rangle$ the free associative algebra generated by $X$ over $F$ and its elements are called \textit{polynomials}.
Let $T(A)$ be the intersection of the kernels of all homomorphisms $F\langle X\rangle\rightarrow A$. Then $T(A)$ is a two-sided ideal of $F\langle X\rangle$
and its elements are called \textit{polynomial identities} of the algebra $A$. If $T(A)$ is not trivial, then $A$ is said to be a \textit{PI-algebra}.
Note that $T(A)$ is stable under the action of any endomorphism of $F\langle X\rangle$.
Any ideal of the algebra $F\langle X\rangle$  which satisfies such property is said to be a \textit{T-ideal}.
Clearly, any T-ideal $I$ is the ideal of the polynomial identities of the algebra $F\langle X\rangle/I$.
From a celebrated theorem of Kemer \cite{Ke}, it is know that in characteristic zero every
$T$-ideal is finitely generated.\\
\indent
If $n\in\mathbb{N}$, then the vector space
\[
P_n:=\textnormal{span}_F \{ x_{\sigma(1)}\cdots x_{\sigma(n)}\mid \sigma\in S_{n}\}
\]
is called the space of \textit{multilinear polynomials of degree $n$}.\\
Since the characteristic of the field $F$ is zero, the standard process of multilinearization shows that $T(A)$ is generated, as a $T$-ideal,
by the subspaces $P_n\cap T(A)$. Actually, it is more efficient to study the quotient space
\[
P_n(A):= P_n/(P_n\cap T(A)).
\]
When $n$ is sufficiently large, the dimension of  $P_n\cap T(A)$ grows factorially, but Regev proved in \cite{R1} that the dimension of $P_n(A)$ grows at most exponentially.
An effective tool to the study of $P_n(A)$ is provided by the representation theory of the symmetric group.
Indeed, one can notice that $P_n(A)$ is an $S_n$-module with the natural left action
and $P_n\cap T(A)$ is an $S_n$-submodule of $P_n$. Hence $P_n(A)$ is an $S_n$-module too.
We shall denote by $\chi_n(A)$  the $S_n$-character of $P_n(A)$ and call it as $n$-th \textit{cocharacter of} $A$.\\
Since we are working in characteristic zero, the representations of symmetric group are totally reducible. Hence
\[
\chi_n(A)=\sum\limits_{\lambda\vdash n} m_{\lambda}(A)\chi_{\lambda}
\]
where the summation runs on all partitions $\lambda$ of $n$, $\chi_{\lambda}$ is the corresponding irreducible character of the symmetric group $S_n$
and $m_{\lambda}(A)\in\mathbb{Z}_{\geq 0}$. The integers $m_{\lambda}(A)$ are called  \textit{multiplicities}.\\
The $n$-th cocharacter $\chi_n(A)$ is related with another important group action,
namely the action of the general linear group $GL_d=GL_d(F)$ on the relatively free algebra $F_d(A)$ of rank $d$ in the variety of algebra $\mathcal{V}(A)$, where
\[
F_d(A)= F\langle X_d\rangle/ (F\langle X_d\rangle\cap T(A))= F\langle x_1,\ldots, x_d\rangle/  (F\langle x_1,\ldots, x_d\rangle\cap T(A)).
\]
The algebra $F_d(A)$ is $\mathbb{Z}^d$-graded with grading defined by
\[
\deg(x_1)=(1,0,\ldots,0),\deg(x_2)=(0,1,\ldots,0),\ldots, \deg(x_d)=(0,0,\ldots,1).
\]
The Hilbert series of $F_d(A)$
\[
H(A;\mathrm{T}_d):=\sum\limits_{n_i\geq 0}\dim(F_d^{(n_1,\ldots,n_d)}(A))t_1^{n_1}\cdots t_d^{n_d}
\]
where $F_d^{(n_1,\ldots, n_d)}(A)$ is the homogeneous component of multi-degree $(n_1,\ldots, n_d)$ of $F_d(A)$,
is a symmetric function which plays the role of the character of the corresponding $GL_d$-representation.
The Schur functions $S_{\lambda}(\mathrm{T}_d)$ are the characters of the irreducible $GL_d$-submodules $W_d(\lambda)$ of $F_d(A)$, therefore
\[
H(A;\mathrm{T}_d)=\sum\limits_{\lambda} m'_{\lambda}(A)S_{\lambda}(\mathrm{T}_d),
\]
where the summation runs on all partitions $\lambda$ with no more than $d$ parts.  By a result of Berele \cite{B1}  and Drensky \cite{D1},
$m'_{\lambda}(A)=m_{\lambda}(A)$. Hence, in principle, if  we know the Hilbert series $H(A;\mathrm{T}_d)$,
we can find the  multiplicities $m_{\lambda}(A)$ for those $\lambda$ with no more than $d$ parts.\\
Therefore, if $A$ is a PI-algebra, we have
\[
H(A;\mathrm{T}_d)=\sum\limits_{\lambda} m_{\lambda}(A)S_{\lambda}(\mathrm{T}_d),
\]
where $\lambda$ has at most $d$ parts and $m_{\lambda}(A)$ is the multiplicity corresponding to $\chi_{\lambda}$ in the cocharacter sequence of $A$.\\

The following proposition expresses the Hilbert series of the product of two $T$-ideals in terms of the Hilbert series of the factors,
and gives the corresponding relations for the Hilbert series of the relatively free algebras.

\begin{proposition}[Formanek \cite{for}, Halpin \cite{hal}]\label{Formanek}
Let $A_1$, $A_2$ and $A$ be PI-algebras such that $T(A)=T(A_1)T(A_2)$.
\begin{itemize}
\item[\rm(i)] Then the Hilbert series of $T(A_1)$, $T(A_2)$ and $T(A)$ are related by
\begin{eqnarray*}
H(F\langle X_d\rangle; \mathrm{T}_d) H(F\langle X_d\rangle\cap T(A);\mathrm{T}_d)
= H(F\langle X_d\rangle\cap T(A_1); \mathrm{T}_d) H(F\langle X_d\rangle\cap T(A_2); \mathrm{T}_d).
\end{eqnarray*}
\item[\rm(ii)] The Hilbert series of $F_d(A)$, $F_d(A_1)$, $F_d(A_2)$ are related by
\begin{eqnarray*}
H(A;\mathrm{T_d}) =  H(A_1;\mathrm{T_d})+ H(A_2;\mathrm{T_d})+ (t_1+\cdots + t_d-1) H(A_1;\mathrm{T_d})H(A_2;\mathrm{T_d}).
\end{eqnarray*}
\end{itemize}
\end{proposition}

\begin{corollary}\label{hilcoro}
Let $A$ and $A_1$ be PI-algebras such that $T(A)=T(A)^n$. Then
\[
H(A;\mathrm{T_d})=\sum\limits_{j=1}^n  \binom{n}{j} (t_1+\cdots + t_d-1)^{j-1}H(A_1;\mathrm{T_d})^j.
\]
\end{corollary}

For a PI-algebra $A$ we define \textit{the multiplicity series of $A$} in $d$ variables by
\[
M(A;\mathrm{T}_d)=\sum\limits_{\lambda} m_{\lambda}(A)\mathrm{T}_d^{\lambda}=\sum\limits_{\lambda}m_{\lambda}(A)t_1^{\lambda_1}\cdots t_d^{\lambda_d}.
\]
Notice that if we know the multiplicity series of $A$ it is possible to find the multiplicities $m_{\lambda}(A)$, where $\lambda$ is a partition in no more than $d$ parts.\\

\subsection{Infinite dimensional Grassmann algebra}

\begin{definition}
Let $W$ be an infinite dimensional vector space  with basis $\{e_1,e_2,\ldots\}$ over a field $F$ of characteristic $0$.
The \textit{Grassmann} (or \textit{exterior}) \textit{algebra} $E=E(W)$ is the associative algebra generated by $\{e_1,e_2,\ldots\}$ and with defining relations
$e_ie_j+e_je_i=0$
for all $i,j\in\mathbb{N}$.
\end{definition}
Observe that $E=E^{(0)}\oplus E^{(1)}$ where
\[
E^{(0)}:= \spn_F \{1,e_{i_1}\cdots e_{i_{2k}}\mid 1\leq i_1< i_2<\cdots< i_{2k}\ k>0\},
\]
\[
E^{(1)}:= \spn_F \{e_{i_1}\cdots e_{i_{2k+1}}\mid 1\leq i_1<\cdots< i_{2k+1}\ k\geq 0\}.
\]
It is easily checked that $E^{(0)}E^{(0)}+ E^{(1)}E^{(1)}\subseteq E^{(0)} $ and $E^{(0)}E^{(1)}+ E^{(1)}E^{(0)}\subseteq E^{(1)}$.
Hence the decomposition $E=E^{(0)}\oplus E^{(1)}$ is a $\mathbb{Z}_2$-grading of $E$. Notice that $E^{(0)}$ coincides with the center of $E$.\\
\indent
The next fact is well known.

\begin{proposition}
The Grassmann algebra $E$ satisfies the polynomial identity
\[
[[x_1, x_2],x_3],
\]
where $[\cdot, \cdot]$ is the Lie commutator, i.e. $[w,y]:=wy-yw$ for any $w,y\in F\langle X\rangle$.
\end{proposition}
The triple commutator is the only generator of $T(E)$ when the field is infinite of characteristic different from $2$.
This fact in characteristic zero was proved by Krakowski and Regev in \cite{KR}. It follows also from the results of Latyshev in \cite{Ly3} and \cite{Ly4}
(but not stated explicitly there). In positive characteristic, the reader can find the proof in \cite{GK}.

\begin{theorem}
The $T$-ideal of $E$ is generated by the polynomial
\[
[x_1, x_2, x_3].
\]
\end{theorem}

The following theorem of Olsson and Regev gives the cocharacter sequence of $E$.

\begin{theorem}[Olsson and Regev \cite{OR}]\label{teo KRG}
Let $E$ be the infinite dimensional Grassmann algebra over a field of characteristic zero. Then the cocharacter sequence of $E$ for any $n\geq 1$ is given by
\[
\chi_n(E)=\sum_{p=1}^n\chi_{(p,1^{n-p})}.
\]
\end{theorem}

The following result gives us an expression for the Hilbert series of $F_d(E)$. The proof of this can be found in \cite{D2}.

\begin{proposition}\label{seriedeE}
Let $E$ the infinite dimensional Grassmann algebra over a field of characteristic zero. The Hilbert series of $F_d(E)$ in $d$ variables is given by
\[
H(E;\mathrm{T}_d)=  \dfrac{1}{2}+\dfrac{1}{2}\displaystyle\prod\limits_{i=1}^d\dfrac{1+t_i}{1-t_i}.
\]
\end{proposition}
The next theorem talks about the polynomial identities of
the $F$-algebra $UT_n (E)$ of $n\times n$ upper triangular matrices with entries in the Grassmann algebra $E$. See also \cite{cenviv1} for the case of $UT_2(E)$ in positive characteristic.

\begin{theorem}[Abakarov \cite{Ab}]\label{tidealutn} The $T$-ideal of $UT_n(E)$ is generated by the polynomial
\[
[x_1,x_ 2,x_3]\cdots [x_{3n-2},x_{3n-1},x_{3n}].
\]
\end{theorem}

\section{The operator $\widehat{Y}$}
In this section, we shall talk about the tools used  
in the development of an algorithm to calculate the multiplicities in the cocharacter sequence of the algebra $UT_n (E)$ of $n\times n$ upper triangular matrices
with entries in the Grassmann algebra $E$ over a field $F$ of characteristic zero. \\
\indent
We follow the ideas developed in \cite{BD}. We start studying the action of two basic operators ($Y$ and $\widehat{Y}$ in the text) in the language of multiplicity series.
Then, we compute the Hilbert series of $UT_n(E)$ 
and find an expression for its multiplicity series. We shall
give a description of the partitions $\lambda$ such that the multiplicities $m_{\lambda}$ are nonzero in the cocharacter sequence of $UT_n(E)$.
Finally, we compute the multiplicity $m_{\lambda}$ for $UT_n(E)$, where  $1\leq n\leq 3$ and $\lambda$ is a partition in no more than $2$ parts.

\begin{definition}\label{Youngoperator}
Let $Y$ be the linear operator in $\mathbb{C}[[\mathrm{V}_d]]$ which sends the multiplicity series of a symmetric function to  the multiplicity series of its Young-derived.
That is, if $g(T_d)$ is a symmetric function, then
\[
Y(M(g);\mathrm{T}_d)=M\left(\left(\prod\limits_{i=1}^d\dfrac{1}{(1-t_{i})}\right)g(\mathrm{T}_{d}); \mathrm{T}_d\right).
\]
The operator $Y$ is called the {\it Young-derived operator}.
\end{definition}

\begin{definition}\label{Yparnovo}
If $g(\mathrm{T}_d)\in\mathbb{C}[[\mathrm{T}_d]]^{S_d}$, then we define the linear operator $\widehat{Y}$ in $\mathbb{C}[[\mathrm{V}_d]]\subseteq\mathbb{C}[[\mathrm{T}_d]]$ as
\[
\widehat{Y}(M(g); \mathrm{T}_d):= M\left( g(\mathrm{T}_d)\left[\dfrac{1}{2}\prod\limits_{i=1}^d (1-t_i)+\dfrac{1}{2}\displaystyle\prod\limits_{i=1}^d(1+t_i)\right]; \mathrm{T}_d\right).
\]
\end{definition}

The following proposition is well known. It describes the multiple action of $Y$ on $1$. A proof can be found in \cite{BD}.

\begin{proposition}
For $d\geq k \geq 1$ the following decomposition holds
\[
\prod\limits_{i=1}^d \dfrac{1}{(1-t_i)^k}=\sum\limits_{\mu}\eta_{\mu} S_{\mu}(\mathrm{T}_d),
\]
where the summation is on all partitions $\mu=(\mu_1,\ldots,\mu_k)$ and
\[
\eta_{\mu}= S_{\mu}\underbrace{(1,\ldots,1)}_{k\ \textrm{times}}=\dim W_k(\mu).
\]
Equivalently, for $k\geq 1$
\[
\]
\end{proposition}
In the general case there is an easy formula which translates the action of $Y$ on $g(\mathrm{T}_d)$ in the language of its multiplicity series.

\begin{proposition}[Drensky and Genov \cite{DG1}]\label{Youngoperatormult}
Let $g(\mathrm{T}_d)\in \mathbb{C}[[\mathrm{T}_d]]^{S_d}$. Then
\begin{equation*}
Y(M(g;\mathrm{T}_d))= \prod\limits_{i=1}^d \dfrac{1}{1-t_i}\sum (-t_2)^{\varepsilon_2}\cdots (-t_d)^{\varepsilon_d}
M(g; t_1t_2^{\varepsilon_2},t_2^{1-\varepsilon_2}t_3^{\varepsilon_3},\ldots, t_{d-1}^{1-\varepsilon_{d-1}}t_d^{\varepsilon_d}, t_d^{1-\varepsilon_d}),
\end{equation*}
where the summation runs on all $\varepsilon_2,\ldots,\varepsilon_d\in\{0,1\}$.
\end{proposition}

Consider the operator $\widehat{Y}$ and notice that
\[
\widehat{Y}^{j}(M(1); \mathrm{T}_d)= \left(\dfrac{1}{2}\prod\limits_{i=1}^d (1-t_i)+\dfrac{1}{2}\displaystyle\prod\limits_{i=1}^d(1+t_i)\right )^j.
\]
Now, we want to know which Schur functions participate in the decomposition of $\widehat{Y}^{j}(M(1); \mathrm{T}_d)$,
that is we want to express $\widehat{Y}^{j}(M(1); \mathrm{T}_d)$ as a linear combination of Schur functions. The next is a direct consequence of the Young rule.

\begin{proposition}\label{Ynovo}
Let $j\geq 1$ and
\[
\left(\dfrac{1}{2}\prod\limits_{i=1}^d (1-t_i)+\dfrac{1}{2}\displaystyle\prod\limits_{i=1}^d(1+t_i)\right )^j =\sum\limits_{\rho} m_{\rho} S_{\rho}(\mathrm{T}_d).
\]
If $m_{\rho}\neq 0$, then
$\rho=(j^{s_1},(j-1)^{s_2},\ldots , 1^{s_j})\vdash 2(n_1+\cdots+n_j)$  for some $n_1,\ldots, n_j$ integers and $(s_1,\ldots, s_j)\in{\mathbb{Z}}^j_{\geq 0}$.
Equivalently, if $m_{\rho}\neq 0$, then $\rho$ has at most $j$ columns.
\end{proposition}
The following lemma allows us to describe $M'(f(\mathrm{T}_d)S_{(1^2)}(T_d); \mathrm{V}_d)$
in terms of the multiplicity series  $M'(f( \mathrm{T}_d); \mathrm{V}_d)$, where $f(\mathrm{T}_d)$ is a symmetric function.

\begin{lemma}\label{d=2}
Let $f(\mathrm{T}_d)\in\mathbb{C}[[\mathrm{T}_d]]^{S_d}$. Then
\begin{equation*}
\begin{split}
M'(f(\mathrm{T}_d)S_{(1^2)}(\mathrm{T}_d);\mathrm{V}_d)= &v_2M'(f(\mathrm{T}_d;\mathrm{V}_d)+
 v_1\sum\limits_{j=2}^{d-1}\frac{v_{j+1}}{v_j}g_j((M'(f(\mathrm{T}_d); v_1,\ldots, v_d))\\
& + \sum\limits_{i=1}^{d-2}\frac{v_{i+2}}{v_i} g_i((M'(f(\mathrm{T}_d); v_1,\ldots, v_d))\\
& + \sum\limits_{\substack{1\leq i,j\leq d-1\\ i+1<j} }\frac{v_{i+1}}{v_i}\frac{v_{j+1}}{v_j}g_{ij}((M'(f(\mathrm{T}_d); v_1,\ldots, v_d)),
\end{split}
\end{equation*}
where
\[
g_j((M'(f(\mathrm{T}_d); \mathrm{V}_d))=M'(f(\mathrm{T}_d); v_1,\ldots, v_d)-M'(f(\mathrm{T}_d);v_1,\ldots,v_{j-1},0,v_{j+1},\ldots,v_d),
\]
\begin{equation*}
\begin{split}
g_{ij}((M'(f(\mathrm{T}_d);\mathrm{V}_d))= & M'(f(\mathrm{T}_d);v_1,\ldots, v_d)\\
+& M'(f(\mathrm{T}_d); v_1,\ldots,v_{i-1},0,v_{i+1},\ldots,v_j,0,v_{j+1},\ldots,v_d)\\
-& M'(f(\mathrm{T}_d);v_1,\ldots,v_{i-1},0,v_{i+1},\ldots,v_d)\\
-& M'(f(\mathrm{T}_d);v_1,\ldots,v_{j-1},0,v_{j+1},\ldots,v_d).
\end{split}
\end{equation*}
\end{lemma}

\begin{proof}
Notice that it is sufficient to prove the lemma for $f(\mathrm{T}_d)=S_{\mu}(\mathrm{T}_d)$,
where $\mu=(\mu_1,\ldots,\mu_d)$ is a partition in no more than $d$ parts.
Rewriting  $M'(S_{\mu}(\mathrm{T}_d;\mathrm{V}_d))$ as
\[
M'(S_{\mu}(\mathrm{T}_d);\mathrm{V}_d)=v_1^{p_1}\cdots v_d^{p_d}
\]
where $p_i=\mu_i-\mu_{i+1}$, $i=1,\ldots,d-1$ and $p_d=\mu_d$,
by the Young rule, we have that $S_{\mu}(\mathrm{T}_d)S_{(1^2)}(\mathrm{T}_d)$ is a linear combination of $S_{\lambda}(\mathrm{T}_d)$, where
\begin{itemize}
\item $\lambda=(\mu_1+1,\mu_2+1,\mu_3,\ldots,\mu_d)$;
\item $\lambda=(\mu_1+1,\ldots,\mu_j,\mu_{j+1}+1,\ldots,\mu_d)$, if $\mu_j>\mu_{j+1}$ and $j\geq 2$;
\item $\lambda=(\mu_1,\ldots,\mu_i,\mu_{i+1}+1,\mu_{i+2}+1,\ldots,\mu_d)$, if $\mu_i>\mu_{i+1}$;
\item $\lambda=(\mu_1+1,\ldots, \mu_i,\mu_{i+1}+1,\ldots,\mu_j,\mu_{j+1}+1,\ldots,\mu_d)$, if $\mu_i>\mu_{i+1}$, $\mu_j>\mu_{j+1}$ and $i+1<j$.
\end{itemize}
In the language of multiplicity series, this means that $M'(S_{\mu}(\mathrm{T}_d)S_{(1,1)}(\mathrm{T}_d))$ is a linear combination of the following terms
\begin{itemize}
\item $v_2(v_1^{p_1}\cdots v_d^{p_d})=v_2 M'(S_{\mu}(\mathrm{T}_d;\mathrm{V}_d))$;
\item $\displaystyle v_1\frac{v_{j+1}}{v_j}M'(S_{\mu}(\mathrm{T}_d;\mathrm{V}_d))$, if $\mu_j>\mu_{j+1}$ and $j\geq 2$;
\item $\displaystyle \frac{v_{i+2}}{v_i}M'(S_{\mu}(\mathrm{T}_d,\mathrm{V}_d))$, if $\mu_i>\mu_{i+1}$;
\item $\displaystyle \frac{v_{i+1}}{v_i}\frac{v_{j+1}}{v_j}M'(S_{\mu}(\mathrm{T}_d,\mathrm{V}_d))$, if $\mu_i>\mu_{i+1}$, $\mu_j>\mu_{j+1}$ and $i+1<j$.
\end{itemize}
Now, observe that
\[
g_j((M'(f(\mathrm{T}_d);\mathrm{V}_d)) = \left\{ \begin{array}{lcc}
             M'(f(\mathrm{T}_d);\mathrm{V}_d), &   if  & p_j>0, \\
             \\ 0, &  if & p_j=0; \\
             \end{array}
\right.
\]
\[
g_{ij}((M'(f(\mathrm{T}_d);\mathrm{V}_d)) = \left\{ \begin{array}{lcc}
             M'(f(\mathrm{T}_d);\mathrm{V}_d), &   if  & p_i>0,\ p_j>0, \\
             \\ 0, &  for & {\it\textrm{all other cases.} }\\
             \end{array}
\right.
\]
Then the result follows easily.
\end{proof}

By Lemma \ref{d=2}, we have the following corollary.

\begin{corollary}\label{Md2}
If $f(\mathrm{T}_2)\in\mathbb {C}[[\mathrm{T}_2]]^{S_2}$, then
\[
\widehat{Y}(M(f;\mathrm{T}_2))=(1+t_1t_2)M(f; T_2)
=\left[\dfrac{1}{2}\prod\limits_{i=1}^2 (1-t_i)+\dfrac{1}{2}\displaystyle\prod\limits_{i=1}^2(1+t_i)]\right] M(f; T_2).
\]
\end{corollary}
Note that if we want to describe $M' (S_{\lambda}(\mathrm{T}_d)S_{(1^k)}(T_d); \mathrm{V}_d)$ in terms of the multiplicity series  $M'(S_{\lambda}( \mathrm{T}_d)$ where $\lambda$ is a partition in no more than $d$ parts and $k\leq d$ a positive integer, we need expressions obtained from  $M'(S_{\lambda}(\mathrm(T)_d);\mathrm{V}_d)$ making $v_{j_1}=v_{j_2}=\cdots =v_{j_r}=0$ where $1\leq r\leq s$ and $j_1\leq\cdots\leq j_r$.

%



\section{Hilbert series and multiplicity series of $UT_n(E)$}
In this section we give an algorithm to calculate the multiplicities in the cocharacter sequence of $UT_n(E)$. We follow the ideas developed in \cite{BD}.\\
\indent
The next result follows from Corollary \ref{hilcoro}, Theorem \ref{tidealutn} and Proposition \ref{seriedeE}.

\begin{theorem}\label{HUk}
The Hilbert series  $H(UT_n(E); \mathrm{T}_d)$ of the algebra $F_d(UT_n(E))$ is
\begin{eqnarray*}
H(UT_n(E));\mathrm{T}_d) & = &\sum_{j=1}^{n} \binom{n}{j} {\left(\dfrac{1}{2}+\dfrac{1}{2}\displaystyle\prod\limits_{i=1}^d\dfrac{1+t_i}{1-t_i}\right)}^{j}{(t_1+\cdots t_d-1)}^{j-1}.
\end{eqnarray*}
\end{theorem}

\noindent
From Definitions  \ref{Youngoperator} and \ref{Yparnovo} we derive
\begin{eqnarray*}
Y (\widehat{Y}(M(g); T_d))& =  & Y\left(M\left( g(T_d)\left[\dfrac{1}{2}\prod\limits_{i=1}^d (1-t_i)+\dfrac{1}{2}\displaystyle\prod\limits_{i=1}^d(1+t_i)\right]\right); \mathrm{T_d}\right) \\
& = & M\left(\displaystyle\prod\limits_{i=1}^d\dfrac{1}{1-t_i} g(T_d)\left[\dfrac{1}{2}\prod\limits_{i=1}^d (1-t_i)+\dfrac{1}{2}\displaystyle\prod\limits_{i=1}^d(1+t_i)\right]; \mathrm{T_d}\right)\\
& = & M \left(g(\mathrm{T}_d)\left[\dfrac{1}{2}+\dfrac{1}{2}\displaystyle\prod\limits_{i=1}^d\dfrac{1+t_i}{1-t_i}\right]; \mathrm{T}_d\right).
\end{eqnarray*}
Notice that $Y (\widehat{Y}(M(g); T_d))$ is well-defined because $g(\mathrm{T}_d)$
and  $\left(\dfrac{1}{2}+\dfrac{1}{2}\displaystyle\prod\limits_{i=1}^d\dfrac{1+t_i}{1-t_i}\right)$ are symmetric functions.
Hence, $g(\mathrm{T}_d)\left[\dfrac{1}{2}+\dfrac{1}{2}\displaystyle\prod\limits_{i=1}^d\dfrac{1+t_i}{1-t_i}\right]$ is symmetric.
Moreover we have $Y\circ \widehat{Y}=\widehat{Y}\circ Y$ too.\\
\indent
Using the previous result, we can obtain an expression for the multiplicity series of $UT_n(E)$. This is the content of the following corollary.

\begin{corollary}\label{MRK}
The multiplicity series of $UT_n(E)$ is
\[
M(UT_n(E);\mathrm{T}_d))  = \sum_{j=1}^{n} \sum\limits_{q=0}^{j-1} \sum\limits_{\lambda\vdash q}{(-1)}^{j-1-q}\binom{n}{j}\binom{j-1}{q} d_{\lambda}Z^j(\mathrm{T}_d^{\lambda}),
\]
where $d_{\lambda}$ is the degree of the irreducible $S_n$-character $\chi_{\lambda}$, $\mathrm{T}_d=t_1^{\lambda_1}\cdots t_d^{\lambda_d}$ and $Z=Y\circ\widehat{Y}$.
\end{corollary}

\begin{proof} Notice that
\[
{(t_1+\cdots + t_d-1)}^{j-1}= \sum\limits_{q=0}^{j-1}{(-1)}^{j-1-q}\binom{j-1}{q}{(t_1+\cdots +t_d)}^{q}
\]
and expanding the expression of $H(U_k(E);\mathrm{T_d})$ from Proposition \ref{HUk}, we get
\begin{eqnarray*}
H(UT_n(E);\mathrm{T}_d) & = &\sum_{j=1}^{n} \binom{n}{j} {\left(\dfrac{1}{2}+\dfrac{1}{2}\displaystyle\prod\limits_{i=1}^d\dfrac{1+t_i}{1-t_i}\right)}^{j}
\sum\limits_{q=0}^{j-1}{(-1)}^{j-1-q}\binom{j-1}{q}{(t_1+\cdots+ t_d)}^{q}.
\end{eqnarray*}
Using the well-known equality
\[
(t_1+\cdots+t_d)^q= S_{(1)}^{q}(\mathrm{T}_d)=\sum\limits_{\lambda\vdash q}d_{\lambda}S_{\lambda}(\mathrm{T}_d),
\]
where $d_{\lambda}$ is the degree of the irreducible $S_q$-character $\chi_{\lambda}$, we have
\begin{eqnarray*}
H(UT_n(E);\mathrm{T}_d) & = &\sum_{j=1}^{n} \binom{n}{j} {\left(\dfrac{1}{2}+\dfrac{1}{2}\displaystyle\prod\limits_{i=1}^d\dfrac{1+t_i}{1-t_i}\right)}^{j}
\sum\limits_{q=0}^{j-1}{(-1)}^{j-1-q}\binom{j-1}{q}\sum\limits_{\lambda\vdash q}d_{\lambda}S_{\lambda}(\mathrm{T}_d)\\
& = & \sum_{j=1}^{n} \sum\limits_{q=0}^{j-1} \sum\limits_{\lambda\vdash q}{(-1)}^{j-1-q}\binom{n}{j}\binom{j-1}{q} d_{\lambda}
{\left(\dfrac{1}{2}+\dfrac{1}{2}\displaystyle\prod\limits_{i=1}^d\dfrac{1+t_i}{1-t_i}\right)}^{j}S_{\lambda}(\mathrm{T}_d).
\end{eqnarray*}
Indeed, the multiplicity series of $S_{\lambda}(\mathrm{T}_d)$ is
\[
M(S_{\lambda}(\mathrm{T}_d); \mathrm{T}_d)=t_1^{\lambda_1}\cdots t_d^{\lambda_d}=\mathrm{T}_d^{\lambda}.
\]
Then
\[
M\left({\left(\dfrac{1}{2}+\dfrac{1}{2}\displaystyle\prod\limits_{i=1}^d\dfrac{1+t_i}{1-t_i}\right)}^{j}S_{\lambda}(\mathrm{T}_d); \mathrm{T}_{d}\right)
=Z^j (M(S_{\lambda}(\mathrm{T}_d); \mathrm{T}_d))= Z^j(\mathrm{T}_d^{\lambda}).
\]
Hence, the multiplicity series of $UT_n(E)$ is

\[
M(H(UT_n(E);\mathrm{T}_d))  =   \sum_{j=1}^{n} \sum\limits_{q=0}^{j-1} \sum\limits_{\lambda\vdash q}{(-1)}^{j-1-q}\binom{n}{j}\binom{j-1}{q} d_{\lambda}Z^j(\mathrm{T}_d^{\lambda}).
\]
\end{proof}
We want to describe those partitions $\lambda$ such that $m_{\lambda}(U_k(E))\neq 0$. Hence we get a sharpening result on the height of $\lambda$.

\begin{theorem}\label{mRk}
If $m_{\lambda}(UT_n(E))\neq 0$ and $\lambda=(\lambda_1,\ldots,\lambda_d)$, then $\lambda_{n+1}\leq 2n-1$.
\end{theorem}

\begin{proof}
By Theorem \ref{HUk} and in the spirit of Corollary \ref{MRK} the nonzero multiplicities $m_{\lambda}(UT_n(E))$ in the cocharacter sequence of $UT_n(E)$ come from the decomposition
\[
{\left(\dfrac{1}{2}+\dfrac{1}{2}\displaystyle\prod\limits_{i=1}^d\dfrac{1+t_i}{1-t_i}\right)}^{j}{(t_1+\cdots + t_d)}^{q}
={\left(\displaystyle\prod\limits_{i=1}^d\dfrac{1}{1-t_i}\right)}^j{\left(\dfrac{1}{2}\displaystyle\prod\limits_{i=1}^d(1-t_i)
+\dfrac{1}{2}\displaystyle\prod\limits_{i=1}^d(1+t_i)\right)}^{j} {S_{(1)}(\mathrm{T}_d)}^{q},
\]
$j\leq n$ and $q\leq n-1$, as linear combination of Schur functions.

By Proposition \ref{Ynovo} the Schur functions $S_{\pi}(\mathrm{T}_d)$ participating in the product
\[
{\left(\dfrac{1}{2}\displaystyle\prod\limits_{i=1}^d(1-t_i)+\dfrac{1}{2}\displaystyle\prod\limits_{i=1}^d(1+t_i)\right)}^{j}
\]
are indexed by partitions $\pi$ having at most $j\leq n$ columns. By the Branching rule,
the multiplication of $S_{\pi}(\mathrm{T}_d)$ by $S_{(1)}(T_d)$ is a linear combination of $S_{\rho}(\mathrm{T}_d)$
where the diagrams of $\rho$ are obtained from the diagrams of $\pi$ by adding a box. Multiplying $q$ times by $S_{(1)}(\mathrm{T}_d)$
we add to the diagram of $\pi$ no more than $q\leq n-1$ boxes in the first row.\\
\indent
\[
{\left(\dfrac{1}{2}\displaystyle\prod\limits_{i=1}^d(1-t_i)+\dfrac{1}{2}\displaystyle\prod\limits_{i=1}^d(1+t_i)\right)}^{j} {S_{(1)}(\mathrm{T}_d)}^{q},
\]
have at most $j+q\leq 2n-1$ boxes in the first row.\\
\indent
Due to the fact that
\[
{\left(\displaystyle\prod\limits_{i=1}^d\dfrac{1}{1-t_i}\right)}^j=\sum\limits_{n_i\geq 0}S_{(m_1)}(\mathrm{T}_d)\cdots S_{(m_j)}(\mathrm{T}_d),
\]
applying Young rule iteratively, it follows that if the partition $\lambda$ appears in the decomposition of
\[
{\left(\displaystyle\prod\limits_{i=1}^d\dfrac{1}{1-t_i}\right)}^j S_{\rho}(\mathrm{T}_d)
\]
then $\lambda$ is of the type $(m_1,\ldots,m_j,(j+q)^{s_1},\ldots, 1^{s_{j+q}})$ with $m_1,\ldots, m_j,s_1,\ldots,s_{j+q}$ nonnegative. Hence $\lambda_{j+1}\leq j+q$.\\
\indent
Therefore, if $m_{\lambda}(UT_n(E))\neq 0$, then $\lambda_{k+1}\leq 2n-1$.
\end{proof}
It is worth mentioning that results about the cocharacters sequence of $E$ and $U_2(E)$ satisfy the bound given in the previous theorem. Olsson and Regev proved in \cite{OR} that
\[
\chi_n(E)= \sum\limits_{k=0}^{n-1} \chi_{(n-k,1^k)}.
\]
Note that the diagrams of  the partition $\lambda=(n-k,1^k)$ have at most one box in the second column, that is, $\lambda_2\leq 1$. Therefore, it agrees with Theorem \ref{mRk}.\\
 In \cite{C} Centrone proved that
\[
H(UT_2(E); \mathrm{T}_d)=\sum m_{\lambda}(U_2(E))S_{\lambda}(\mathrm{T}_d),
\]
where  $\lambda=(m_1,m_2, 3,2^m, 1^l)$ or $\lambda=(m_1,m_2,2^m, 1^l)$. Hence the diagrams of  partition $\lambda$ have at most 3 boxes in the third row, that is, $\lambda_3\leq 3$.

\section{Some applications of the multiplicity series of $UT_n(E)$}
In this section, we shall compute the multiplicity series of $UT_n(E)$ in two variables for $n\in\{1,2,3\}$. As a consequence, we get the multiplicities $m_{\lambda}$
in the cocharacter sequences of $UT_n(E)$ where $1\leq n\leq 3$ and $\lambda$ is a partition in no more than 2 parts.

\begin{proposition}\label{M2EUTE}
Consider the algebras $E$ and $UT_2(E)$ then
\begin{enumerate}
\item[{\rm(i)}] the multiplicity series of $E$ in two the variables is
\begin{eqnarray}
M'(E;\mathrm{V}_2) & = & \dfrac{1+v_2}{1-v_1}.
\end{eqnarray}
\item[{\rm(ii)}] the multiplicity series of $UT_2(E)$ in two variables is
\begin{eqnarray}\label{M2UT2}
M'(UT_2(E);\mathrm{V}_2) & = & \dfrac{2(1+v_2)}{1-v_1}+ \dfrac{(1+v_2)^2(-1+v_1+2v_2-v_1v_2)}{(1-v_1)^2(1-v_2)}.
\end{eqnarray}
\end{enumerate}
\end{proposition}

\begin{proof} (i) This case follows immediately from the result of Olsson and Regev \cite{OR} because
\[
m_{\lambda}=\begin{cases}1 \text{ for }\lambda=(k,1^{n-k}),\\
0\text{ otherwise}
\end{cases}
\]
and
\[
M(E;\mathrm{T}_2)=\sum_{n\geq 0}t_1^n+\sum_{n\geq 2}t_1^{n-1}t_2.
\]
We shall restate this proof in terms of Corollary \ref{MRK},

By Corollary \ref{MRK}, we get
\begin{eqnarray*}
M(E;\mathrm{T_2}) & = & Z(1;\mathrm{T_2}).
\end{eqnarray*}
Since $Z=Y\circ\widehat{Y}$, applying Proposition \ref{Youngoperatormult} and Corollary \ref{Md2} we have
\[
M(E;\mathrm{T_2})  =  \dfrac{1+t_1t_2}{1-t_1}.
\]
Recall that $v_1=t_1$ and $v_2=t_1t_2$. Hence
\begin{eqnarray*}
M'(E;\mathrm{V}_2) & = & \dfrac{1+v_2}{1-v_1}.
\end{eqnarray*}
(ii) By Corollary \ref{MRK}, we have
\begin{eqnarray*}
M(UT_2(E); \mathrm{T}_2) & = & 2Z(1)-Z^2(1) + Z^2(t_1).
\end{eqnarray*}
Now, we shall compute $2Z(1),Z^2(1)$ and $Z^2(t_1)$ using Proposition \ref{Youngoperatormult} and Corollary \ref{Md2}
\begin{itemize}
\item $2Z(1)=\dfrac{2(1+t_1t_2)}{1-t_1}$;
\item $Z^2(1)=\dfrac{(1+t_1t_2)^2}{(1-t_1)^2(1-t_1t_2)}$;
\item  $Z^2(t_1)=\dfrac{(1+t_1t_2)^2(t_1+2t_1t_2-t_1^2t_2)}{(1-t_1)^2(1-t_1t_2)}$.
\end{itemize}
Hence
\begin{eqnarray*}
M(UT_2(E); \mathrm{T}_2) & = & \dfrac{2(1+t_1t_2)}{1-t_1}-\dfrac{(1+t_1t_2)^2}{(1-t_1)^2(1-t_1t_2)}+ \dfrac{(1+t_1t_2)^2(t_1+2t_1t_2-t_1^2t_2)}{(1-t_1)^2(1-t_1t_2)}\\
& = & \dfrac{2(1+t_1t_2)}{1-t_1}+ \dfrac{(1+t_1t_2)^2(-1+t_1+2t_1t_2-t_1^2t_2)}{(1-t_1)^2(1-t_1t_2)}.
\end{eqnarray*}
Finally, we have
\begin{eqnarray*}
M'(UT_2(E);\mathrm{V}_2) & = & \dfrac{2(1+v_2)}{1-v_1}+ \dfrac{(1+v_2)^2(-1+v_1+2v_2-v_1v_2)}{(1-v_1)^2(1-v_2)}
\end{eqnarray*}
and we are done.
\end{proof}
Now we are able to compute the multiplicities $m_{\lambda}$ in the cocharacter sequences of $E$ and $UT_2(E)$ when $\lambda$ is a partition in no more than $2$ parts.
The goal of this result is to show how to find out the multiplicities using Corollary \ref{MRK} and Proposition \ref{M2EUTE}.

\begin{corollary}\label{mulEUT2}
Let $\lambda$ be a partition in no more than $2$ parts. Then:
\begin{enumerate}
\item[{\rm (i)}]  the multiplicity $m_{\lambda}$ in the cocharacter sequences of $E$ is given by
\[
m_{\lambda}=\begin{cases}
1\text{ if }  \lambda=(n), \\
1\text{ if } \lambda=(\lambda_1,1),\ \lambda_1\geq 1, \\
0\text{ for all other } \lambda.
\end{cases}
\]
\item[{\rm (ii)}] the multiplicity $m_{\lambda}$ in the cocharacter sequence of $UT_2(E)$ is given by
\[
m_{\lambda}= \left\{ \begin{array}{lcl}
             1 &   if &   \lambda=(n), \\
             \\ \lambda_1 &  if &\lambda=(\lambda_1,1),\ \lambda_1\geq 1, \\
              \\ 3\lambda_1-4 &   if  & \lambda=(\lambda_1,2), \lambda_1\geq 2,\\
             \\ 4(\lambda_1-\lambda_2+1) & if & \lambda=(\lambda_1,\lambda_2), \ \lambda_1\geq \lambda_2\geq 3.
             \end{array}
   \right.
\]
\end{enumerate}
\end{corollary}

\begin{proof}
(i) By Proposition \ref{M2EUTE}, we get
\[
M'(E; \mathrm{V}_2)=\sum\limits_{n\geq 0}v_1^n +\sum\limits_{n\geq 0} v_1^nv_2.
\]
From the first summand of the above equality, we have that if $\lambda=(n)$ with $n\geq 0$, then $m_{\lambda}=1$.
Observe that $v_1^nv_2$ with $n\geq 0$ corresponds to the partition $\lambda=(n+1,1)$. It follows that if $\lambda=(\lambda_1, 1)$, where $\lambda_1\geq 1$, then $m_{\lambda}=1$.

(ii) By Proposition \ref{M2EUTE}, it follows that
\begin{eqnarray*}
M'(UT_2(E); \mathrm{V}_2) & = & \sum\limits_{n\geq 0}2v_1^n +\sum\limits_{n\geq 0} 2v_1^nv_2-\sum\limits_{m,n\geq 0} (n+1)v_1^n v_2^m - \sum\limits_{n\geq 0,m\geq 1} 2(n+1)v_1^n v_2^m \\
&  & -\sum\limits_{m\geq 2,n\geq 0} (n+1)v_1^n v_2^m +
 \sum\limits_{m\geq 1,n\geq 0} 2(n+1)v_1^n v_2^m \\
 & & +\sum\limits_{m\geq 2,n\geq 0} 4(n+1)v_1^n v_2^m + \sum\limits_{m\geq 3,n\geq 0} 2(n+1)v_1^n v_2^m \\
 & &  +\sum\limits_{n\geq 1} nv_1^n +  \sum\limits_{n\geq 1} 2nv_1^nv_2+  \sum\limits_{n\geq 1} nv_1^nv_2^2.\\
\end{eqnarray*}
Therefore
\begin{equation*}\label{sM2}
\begin{split}
M'(UT_2(E); \mathrm{V}_2) = & \sum\limits_{n\geq 0} v_1^n+ \sum\limits_{n\geq 0}(n+1)v_1^nv_2+   \sum\limits_{n\geq 0} (3n+2) v_1^nv_2^2 \\
 & +\sum\limits_{n\geq 0,m\geq 3} 4(n+1) v_1^nv_2^m.
\end{split}
\end{equation*}
First, consider the first summand of the above equality  and observe that $v_1^{n}$ corresponds to the partition $\lambda=(n)$. So, if $\lambda=(n)$, then $m_{\lambda}=1$.\\
\indent
Now, notice that there is a one-to-one correspondence between monomials $v_1^nv_2$ with $n\geq 0$ and partitions $\lambda=(n+1,1)$. Hence the equality gives that if $\lambda=(n,1)$ where $n\geq 1$ then $m_{\lambda}=n$.\\
\indent
Finally, we have that $v_1^nv_2^m$ corresponds to the partition $\lambda=(n+m,m)$. Then $m_{\lambda}=4(n+1)=4((n+m)-m+1)$
and it follows that if $\lambda=(\lambda_1,\lambda_2)$ with $\lambda_1\geq\lambda_2\geq 3$, then $m_{\lambda}= 4(\lambda_1-\lambda_2 +1)$.\\
\indent
The case $n\geq 0$, $m\geq 3$ is treated similarly.
\end{proof}

We highlight that Corollary $\ref{mulEUT2}$ agrees with the results presented in \cite{OR} and \cite{C} when the partitions have no more than two parts. \\
\indent
Now, we shall compute the multiplicity series of $UT_3(E)$ in two variables.

\begin{theorem}\label{MOUT3}
\begin{enumerate}
\item[{\rm (i)}] The multiplicity series of $UT_3(E)$ in two variables is
\begin{equation}
\begin{split}
M'(UT_3(E);\mathrm{V}_2)= & \dfrac{3(1+v_2)}{1-v_1}-\dfrac{3(1+v_2)^2}{(1-v_1)^2(1-v_2)}+ 3\left(\dfrac{2(1+v_2)^2}{(1-v_1)^2(1-v_2)}+ \dfrac{v_1(1+v_2)^2}{(1-v_1)^2}\right)\\
& +\dfrac{1}{(1-v_1)^3(1-v_2)^3}\left(1-2v_1+v_1^2-2v_2+2v_1v_2-4v_2^2+8v_1v_2^2\right.\\
&  -3v_1^2v_2^2+7v_2^3
 -5v_1v_2^3 + 10v_2^4-13v_1v_2^4+ 3v_1^2v_2^4 -v_2^5-v_1v_2^5\\
 &\left. -3v_2^6+3v_1v_2^6-v_1^2v_2^6\right).
\end{split}
\end{equation}
\item[{\rm (ii)}] Let $\lambda$ be a partition in not more than 2 parts. Then the multiplicities $m_{\lambda}$ in the cocharacter sequence of $UT_3(E)$ are given by
{\small
\[
m_{\lambda}= \left\{ \begin{array}{lcl}
             1 &   if &   \lambda=(n), \\
             \\ \lambda_1 &  if &\lambda=(\lambda_1,1),\ \lambda_1\geq 1, \\
             \\ \dfrac{1}{2}(\lambda_1+2)(\lambda_1-1) &   if  & \lambda=(\lambda_1,2), \lambda_1\geq 2,\\
              \\
               \dfrac{1}{2}(16-17\lambda_1+5\lambda_1^2) &   if  & \lambda=(\lambda_1,3), \lambda_1\geq 3,\\
             \\ 14-16\lambda_2+ 4\lambda_2^2 + 2(\lambda_1-\lambda_2)(2-5(\lambda_1-\lambda_2))&&\\ \\
             + 4\lambda_2(\lambda_1-\lambda_2)(-3+ \lambda_2+ (\lambda_1-\lambda_2)) & if & \lambda=(\lambda_1,\lambda_2), \ \lambda_1\geq \lambda_2\geq 4.\\
             \end{array}
                \right.
\]
}
\end{enumerate}
\end{theorem}

\begin{proof} (i)
By Corollary \ref{MRK}, we have
\begin{equation*}
M(UT_3(E); \mathrm{T}_2)=\ 3Z(1)-3Z^2(1)+ 3Z^2(t_1)+ Z^3(1)-2Z^3(t_1)+ Z^3(t_1^2)+Z^3(t_1t_2).
\end{equation*}
Due to Corollary \ref{Md2} and Proposition \ref{Youngoperatormult}, we get
\begin{equation}\label{MPUT3}
\begin{split}
3Z(1)-3Z^2(1)+ 3Z^2(t_1)=& \dfrac{3(1+v_2)}{1-v_1}-\dfrac{3(1+v_2)^2}{(1-v_1)^2(1-v_2)}\\
& +3\left(\dfrac{2(1+v_2)^2}{(1-v_1)^2(1-v_2)}+ \dfrac{v_1(1+v_2)^2}{(1-v_1)^2}\right),
\end{split}
\end{equation}
\begin{equation}\label{MP2UT3}
\begin{split}
Z^3(1)-2Z^3(t_1)+ Z^3(t_1^2)+Z^3(t_1t_2)= & \dfrac{1}{(1-v_1)^3(1-v_2)^3}\left(1-2v_1+v_1^2-2v_2+2v_1v_2\right. \\
& -4v_2^2+8v_1v_2^2 -3v_1^2v_2^2+7v_2^3
 -5v_1v_2^3 + 10v_2^4\\
 &\left. -13v_1v_2^4+ 3v_1^2v_2^4 -v_2^5-v_1v_2^5
  -3v_2^6+3v_1v_2^6-v_1^2v_2^6\right).
\end{split}
\end{equation}
Now, the result follows.

(ii)
We have expanded the expression of
$M'(UT_3(E);\mathrm{V}_2)$  given in the part (i) into a power series
using the following  well known equalities:
\begin{equation*}
\dfrac{v_1^{a_1}v_2^{a_2}}{1-v_1}= \sum\limits_{n\geq a_1}v_1^nv_2^{a_2},
\end{equation*}
\begin{equation*}
\dfrac{v_1^{a_1}v_2^{a_2}}{(1-v_1)^2}= \sum\limits_{n\geq a_1}(n-a_1+1)v_1^{n}v_2^{a_2},
\end{equation*}
\begin{equation*}
\dfrac{v_1^{a_1}v_2^{a_2}}{(1-v_1)^2(1-v_2)}=\sum\limits_{n\geq a_1}\sum\limits_{m\geq a_2}(n-a_1+1)v_1^{n}v_2^{m},
\end{equation*}
\begin{equation*}
\dfrac{v_1^{a_1}v_2^{a_2}}{(1-v_1)^3(1-v_2)^3}=\sum\limits_{n_\geq {a_1}}\sum\limits_{m\geq a_2} {n-a_1+2\choose 2}{m-a_2+2\choose 2}v_1^{n}v_2^{m}.
\end{equation*}
Easy manipulations give the explicit expression for $m_{\lambda}$ where $\lambda$ is a partition in no more than two parts.
In particular, if we want to compute the multiplicity of $\lambda=(\lambda_1,1)$, we need to study  the terms of type $v_1^nv_2$ in $M'(UT_{3};\mathrm{V}_2)$. Hence we shall study the following expression
\begin{equation*}\label{M3}
\begin{split}
\dfrac{3v^2}{1-v_1}-\left(\dfrac{3}{(1-v_1)^2(1-v_2)}+\dfrac{6v_2}{(1-v_1)^2(1-v_2)}\right)+ \dfrac{6v_1v_2}{(1-v_1)^2}\\
 +\dfrac{6v_2}{(1-v_1)^2(1-v_2)}+ \dfrac{1}{(1-v_1)^3(1-v_2)^3}\left(1-2v_1+v_1^2-2v_2+2v_1v_2\right).
\end{split}
\end{equation*}
Notice that if $n\geq 2 $ and $m=1$, then $\lambda=(n+1,1)$. By the last expression, we get
\begin{equation*}
\begin{split}
m_{\lambda}= & 3-(3(n+1)+6(n+1))+6n+6(n+1)+\dfrac{6(n+1)(n+2)}{4}-\dfrac{6n(n+1)}{2}\\
& + \dfrac{6(n-1)n}{4}-\dfrac{2(n+2)(n+1)}{2}+ \dfrac{2n(n+1)}{2}\\
= & n+1.
\end{split}
\end{equation*}
 Equivalently, if ${\lambda}=(n,1)$, with $n\geq 3$, then $m_{\lambda}=n$. Observe that, if $m=0$ and $n=1$, then $\lambda=(1,1)$ and $m_{\lambda}=1$.
 Finally, if $n=1$ and $m=1$, then $\lambda=(2,1)$. By the last expression, $m_{\lambda}=2$.\\
 \indent
We conclude that, if $\lambda=(\lambda_1,1)$, with $\lambda_1\geq 1$, then $m_{\lambda}=\lambda_1$.
The other cases are treated similarly and the proof follows.
\end{proof}

\section{Double Hilbert series and hook-Schur functions}
Notice that the multiplicity series of $UT_n(E)$ given in Corollary \ref{MRK} gives us only the multiplicities $m_{\lambda}$ when $\lambda$ has no more than $d$ parts.
Our next goal is to find an algorithm that allows us to calculate the multiplicities in the cocharacter sequence of $UT_n(E)$ having more ``freedom" in the partition $\lambda$,
and to find $m_{\lambda}$ without any restriction of the height of $\lambda$. For this purpose we shall introduce some new concepts.\\

Consider the infinite dimensional Grassmann algebra $E$, as pointed out above.
Then
\[
{\mathcal B}=\{1, e_{i_1}\cdots e_{i_m}\mid i_1<\cdots<i_m,\ m=1,2,\ldots\}
\]
is a basis of $E$.
We recall the action $*$ of $S_n$ introduced by Olsson and Regev \cite{OR} and used by Berele and Regev in \cite{BR2}.
Given $1\neq a=e_{i_1}\cdots e_{i_m}\in\mathcal B$, we write $l(a)=m$. Let $(a)=(a_1,\ldots,a_n)$, where
$a_1,\ldots,a_n\in\mathcal B$, and define
\[
I=\text{\rm Odd}(a)=\{i\mid l(a_i)\equiv 1\ \text{(mod 2)} \}.
\]

\begin{remark}
Let $I\subseteq\{1,\ldots,n\}$ ($I$ is possibly empty), $\sigma\in S_n$. Choose any $(a)=(a_1,\ldots,a_n)$, $a_i\in\mathcal B$, such that
$a_1\cdots a_n\neq 1$ and $\text{\rm Odd}(a)=I$. Then in $E$
\[
a_{\sigma(1)}\cdots a_{\sigma(n)}=\pm a_1\cdots a_n
\]
Note that the sign $\pm$ depends on $I$ and $\sigma$ but does not depend on the concrete choice of $a$.
\end{remark}

\begin{definition}
Let $I\subseteq\{1,\ldots,n\}$ ($I$ is possibly empty), $\sigma\in S_n$. Choose any $n$-tuple $(a)=(a_1,\ldots,a_n)$, $a_i\in\mathcal B$, such that
$a_1\cdots a_n\neq 0$ and $\text{\rm Odd}(a)=I$.
We define $f_I(\sigma)=\pm 1$ by the equality
\[
a_{\sigma(1)}\cdots a_{\sigma(n)}=f_I(\sigma)a_1\cdots a_n.
\]
\end{definition}

\begin{definition}
Fixing two noncommuting sets of variables $X=\{x_1,\ldots,x_k\}$ and $Z=\{z_1,\ldots,z_l\}$ and a vector space $V$ with basis
$X\cup Z=\{x_1,\ldots,x_k,z_1,\ldots,z_l\}$,
the tensors $v_1\otimes\cdots\otimes v_n$, $v_i\in X\cup Z$, form a basis of $V^{\otimes n}$. Given such $(v)=v_1\otimes\cdots\otimes v_n$, we define the $Z$-indices of $(v)$
by $IZ(v)=\{i\mid v_i\in Z\}$. Let $\sigma\in S_n$ and let us define the right action $*$ of $\sigma$ by
\[
(v_1\otimes\cdots\otimes v_n)* \sigma= f_{IZ(v)} (\sigma) v_{\sigma(1)}\otimes\cdots v_{\sigma(n)}.
\]
Finally, extend the action $*$ of $\sigma$ to the whole $V^{\otimes n}$ by linearity.
As usual, we may identify the vector space $P_n$ of multilinear polynomials of degree $n$ with the group algebra $FS_n$
and define an action $*$ of $P_n$ on $V^{\otimes n}$.
\end{definition}

Let us consider now the double Hilbert series (or the double Poincar\'e series)
related to the polynomial identities of a PI-algebra $A$.
In the above setup we identify
the tensor algebra of $V$ by $T_V$ and the free algebra \[F\langle X,Z\rangle=F\langle x_1,\ldots,x_k,z_1,\ldots,z_l\rangle.\]
The latter algebra is a free superalgebra assuming, as usual, that $x_1,\ldots,x_k$
and $z_1,\ldots,z_l$ are, respectively, the even and the odd free generators.
Let
\[
\langle a;b\rangle=\langle a_1,\ldots,a_k;b_1,\ldots,b_l\rangle,
\]
where $a_1+\cdots+a_k+b_1+\cdots+b_l=n$, and let $V\langle a;b\rangle\subseteq V^{\otimes n}$ being the subspace of all
polynomials which are homogeneous in each of $x_1,\ldots,x_k,z_1,\ldots,z_l$ of degree $a_i$ in $y_i$ and $b_j$ in $z_j$.

\begin{definition}
Let $A$ be a PI-algebra and consider the following set of commutative variables  $T_k=\{t_1,\ldots , t_k\}$, $Y_l=\{y_1,\ldots, y_l\}$. The double Hilbert series
of $A$ is defined to be
\[
H(A; \mathrm{T}_{k}, \mathrm{Y}_{l})= H(A; t_1,\ldots, t_k;y_1,\ldots, y_l):=\sum_{\langle a;b\rangle}\dim_F(V\langle a;b\rangle/V\langle a;b\rangle* Q_n)t_1^{a_1}\cdots t_k^{a_k}y_1^{b_1}\cdots y_l^{b_l},
\]
where $Q_n=T(A)\cap P_n$.
\end{definition}

Note that the variables $t$'s and $y$'s count, respectively, the degrees of the $x$'s and $z$'s.\\

There is another way to define double Hilbert series which is an exact analog of the definition of Hilbert series of relatively free algebras.
We recall that if $A$ is a PI-algebra, then $A^M:=A\otimes_F E$ inherits the superalgebra structure
from the natural ${\mathbb Z}_2$-grading of $E$, i.e, $A^{(0)}=A\otimes E^{(0)}$ and $A^{(1)}=A\otimes E^{(1)}$.\\

If $T_2(A^M)\subseteq F\langle x_1,x_2,\ldots,z_1,z_2,\ldots\rangle$
is the $T_2$-ideal of the ${\mathbb Z}_2$-graded polynomial identities of $A^M$,
then the relatively free ${\mathbb Z}_2$-graded algebra
\[
F\langle x_1,\ldots,x_k,z_1,\ldots,z_l\rangle/(T_2(A^M)\cap F\langle x_1,\ldots,x_k,z_1,\ldots,z_l\rangle)
\]
is called the \textit{magnum} of $A$. For more details on the magnum of a PI-algebra see \cite{B2}. The following result is well known (see \cite{B2}) and gives that the double Hilbert series related to the PI-algebra
$A$ coincides with the Hilbert series of the magnum of $A$.

\begin{proposition}
Let $A$ be a PI-algebra. If $\langle a;b\rangle=\langle a_1,\ldots,a_k;b_1,\ldots,b_l\rangle$
be such that $a_1+\cdots+a_k+b_1+\cdots+b_l=n$, then $V\langle a;b\rangle*Q_n=V\langle a;b\rangle\cap T_2(A^M)$.
\end{proposition}
Now, we are going to talk about hook Schur functions and its relations with the double Hilbert series of a PI-algebra. We shall begin with a definition that generalizes the concept of a semistandard tableau.

\begin{definition}
Fix integers $k,l\geq 0$ such that $k+l>0$ and $k+l$ variables $t_1,\ldots, t_k, y_1,\ldots y_l$, so that $t_1<\cdots <t_k<y_1<\cdots y_l$. Let $\lambda$ be a partition with Young diagram $D_{\lambda}$. Fill $D_{\lambda}$ with elements from $\{t_1,\ldots, t_k, y_1,\ldots y_l\}$, allowing repetitions, to get a $(k,l)$-tableau $T_{\lambda}$. Such $T_{\lambda}$ is said to be \textit{$(k,l)$-semistandard} if
\begin{itemize}
\item[a)] The ``$t$ part" (i.e., the cells filled with $t_i$'s) of $T_{\lambda}$ is a tableau. (Thus the ``$y$ part" is a skew tableau);
\item[b)]  The  $t_i$'s are nondecreasing in rows, strictly increasing in columns;
\item[c)] The $y_j$'s are nondecreasing in columns, strictly increasing in rows.
\end{itemize}
\end{definition}

\begin{definition}\label{hsfunctions}
For a $(k,l)$-semistandard tableau $T_{\lambda}$ we define $w^{T_{\lambda}}= t_1^{a_1}\cdots t_k^{a_k}y_1^{b_1}\cdots y_l^{b_l}$,
where each $a_i$ counts the number of entries of $t_i$ in $T_{\lambda}$  and each $b_j$ counts the number of entries of  $y_j$ in $T_{\lambda}$. So the {\it hook Schur function} is defined by
\[
HS_{\lambda}(\mathrm{T}_k, \mathrm{Y}_{l})=HS_{\lambda}(t_1,\ldots,t_k,y_1,\ldots, y_l)=\sum \{w^{T_{\lambda}}\mid T_{\lambda} \text{ is a $(k,l)$-semistandard}\}.
\]
\end{definition}

Let $H(k,l;n)=\{\lambda=(\lambda_1,\lambda_2,\cdots)\vdash n\mid\lambda_{k+1}\leq l\}$ and
\[
H(k,l)=\displaystyle\bigcup_{n\geq 0} H(k,l;n).
\]
Note that if $\lambda\in H(k,l)$, then the Young diagram  $D_{\lambda}$ lies in the hook of width $k$ of the {\it arm}  and width $l$ of the {\it leg}. It is not hard to see from the definition that  $HS_{\lambda}(\mathrm{T}_k, \mathrm{Y}_{l})\neq 0$ if and only if $\lambda\in H(k,l)$.\\

The following theorem of Amitsur and Regev  shows that by taking $k,l$ large enough we can capture all partitions that have  nonzero multiplicities in the cocharacter sequence of a PI-algebra.

\begin{theorem}[Amitsur and Regev \cite{AR}]\label{existegancho}
If $A$ is a PI-algebra over a field of characteristic zero, then there exist $k$ and $l$ such that the cocharacter sequence of $A$ lies in the $k$ by $l$  hook, i.e.,  if $m_{\lambda}(A)\neq 0$, then $\lambda\in H(k,l)$.
\end{theorem}

Let $A$ be a PI-algebra. We shall write $\chi(A)\subseteq H(k,l)$ when the nonzero multiplicities $m_{\lambda}(A)$ in the cocharacter sequence $\chi_n(A)$, $n=0,1,2\ldots , $ appear  only for those $\lambda\in H(k,l)$. By Theorems 6 and 11 of \cite{BR1} we have the following.

\begin{theorem}[Berele and Regev \cite{BR1}]\label{Berel1}
Let $A$ be a PI-algebra with cocharacter sequence
\[
\chi_n(A)=\sum\limits_{\lambda\vdash n}m_{\lambda}(A)\chi_{\lambda}, \ n\geq 0.
\]
Then there exist nonnegative integers $k$ and $l$ such that
\[
H(A;\mathrm{T}_k,\mathrm{Y}_l)=\sum\limits_{n=0}^{\infty} \sum\limits_{\lambda\in H(k,l;n)}m_{\lambda}(A)HS_{\lambda}(\mathrm{T}_{k}, \mathrm{Y}_{l}).
\]
\end{theorem}

The following result is a generalization of Theorem \ref{Formanek}.

\begin{proposition}\label{HS2}
Let $A_1, A_2$ and $A$ be PI-algebras such that
$T(A)=T(A_1)T(A_2)$. Then
\[
H(A;\mathrm{T}_k,\mathrm{Y}_l)=H(A_1;\mathrm{T}_k,\mathrm{Y}_l)+ H(A_2;\mathrm{T}_k,\mathrm{Y}_l)+ (HS_{(1)}(\mathrm{T}_k,\mathrm{Y}_l)-1)H(A_1;\mathrm{T}_k,\mathrm{Y}_l)H(A_2;\mathrm{T}_k,\mathrm{Y}_l).
\]
\end{proposition}

\begin{proof}
Berele and Regev proved in \cite{BR3} that if  $T(A)=T(A_1)T(A_2)$ then
\begin{equation}\label{br1}
\chi_n (A)=\chi_n(A_1)+\chi_n(A_2)+ \chi_{(1)}\widehat{\otimes}\sum\limits_{j=0}^{n-1} \chi_j(A_1)\widehat{\otimes}\chi_{n-j-1}(A_2)-\sum\limits_{j=0}^n \chi_j(A_1)\widehat{\otimes}\chi_{n-j}(A_2)
\end{equation}
where $\widehat{\otimes}$ denotes the ``outer" tensor product of characters. Recall that for irreducible characters $\widehat{\otimes}$ behaves as in the Littlewood-Richardson rule.\\
\indent
In virtue of Theorem \ref{Berel1} we have
\[
H(A;\mathrm{T}_k,\mathrm{Y}_l)=\sum\limits_{n=0}^{\infty} \sum\limits_{\lambda\in H(k,l;n)}m_{\lambda}(A) HS_{\lambda}(T_{k}, Y_{l}),
\]
\[
H(A_1;\mathrm{T}_k;\mathrm{Y}_l)=\sum\limits_{n=0}^{\infty} \sum\limits_{\alpha\in H(k,l;n)}m_{\alpha}(A_1) HS_{\alpha}(T_{k}, Y_{l}),
\]
\[
H(A_2;\mathrm{T}_k,\mathrm{Y}_l)=\sum\limits_{n=0}^{\infty} \sum\limits_{\beta\in H(k,l;n)}m_{\beta}(A_2) HS_{\beta}(T_{k}, Y_{l}).
\]

Multiplying the hook Schur functions with the Littlewood-Richardson rule (see \cite{BR2}, section 6), the equality (\ref{br1}) implies
\[
H(A;\mathrm{T}_k,\mathrm{Y}_l)=H(A_1;\mathrm{T}_k,\mathrm{Y}_l)+ H(A_2;\mathrm{T}_k,\mathrm{Y}_l)+ (HS_{(1)}(\mathrm{T}_k,\mathrm{Y}_l)-1)H(A_1;\mathrm{T}_k,\mathrm{Y}_l)H(A_2;\mathrm{T}_k,\mathrm{Y}_l),
\]
as desired.
\end{proof}

\begin{corollary}\label{prodhs}
Let $A$ and $C$ be PI-algebras such that
$T(C)=T(A)^m$. Then
\begin{equation*}
H(C;\mathrm{T}_k,\mathrm{Y}_l)=\sum_{j=1}^{m} \binom{m}{j}{H(A; \mathrm{T}_k, \mathrm{Y}_l)}^{j}{( S_{(1)}(\mathrm{T}_k; \mathrm{Y}_l)-1)}^{j-1}.
\end{equation*}
\end{corollary}

\section{The double Hilbert series of $UT_n(E)$}

We know that the cocharacter sequence of $E$ lies in the hook $H(1,1)$.
The double Hilbert series $H(E; t_1,y_1)$ was computed in \cite{BR1}.
We present once again the computation of $H(E; t_1,y_1)$
as a direct application of the definition of hook Schur functions.\\
\indent
We shall also compute $H(E; \mathrm{T}_k,  \mathrm{T}_l)$ for $k,l$ nonnegative integers.
Moreover, we shall find an expression for the double Hilbert series of $UT_n(E)$.
Finally, using $H(UT_n(E);\mathrm{T}_k, \mathrm{Y}_l)$ we shall give a description of the  nonzero multiplicities $m_{\lambda}$ in the cocharacter sequence of $UT_n(E)$.

\begin{proposition}
Let $E$ be the infinite dimensional Grassmann algebra. Then
\[
H(E;t_1, y_1)=\frac{1+t_1y_1}{(1-t_1)(1-y_1)}.
\]
\end{proposition}

\begin{proof}
By Theorem \ref{teo KRG} we know that for any $n\geq 1$, if $\lambda=(p,1^n-p)$, we have $m_{\lambda}(E)=1$.  In light of Theorem \ref{Berel1}, we have to compute $HS_{\lambda}(t_1,y_1)$ to determine $H(E;t_1,y_1)$.
Note that the only $(1,1)$-semistandard tableaux of shape $\lambda$ are
\ytableausetup
{mathmode, boxsize=1em}
\[
\begin{ytableau}
t_1 & & & & & t_1\\
y_1  \\
\ \\
\ \\
\ \\
y_1
\end{ytableau} \hspace{1.5cm} \begin{ytableau}
t_1 & & & &t_1 & y_1\\
y_1  \\
\ \\
\ \\
\ \\
y_1
\end{ytableau}
\]
corresponding to the monomials $t_1^py_1^{n-p}$ and  $t_1^{p-1}y_1^{n-p+1}$, respectively. Hence
\[
H(E;t_1,y_1)= 1+ \sum\limits_{n=1}^{\infty} \sum\limits_{p=1}^{n} (t_1^py_1^{n-p}+ t_1^{p-1}y_1^{n-p+1}).
\]
Note that
\begin{eqnarray*}
1+ \sum\limits_{n=1}^{\infty} \sum\limits_{p=1}^{n} (t_1^py_1^{n-p}+ t_1^{p-1}y_1^{n-p+1}) & = & 1 + (t_1+y_1) \sum\limits_{n=1}^{\infty} \sum\limits_{p=1}^{n} t_1^{p-1}y_1^{n-p}  \\
& = & 1 +(t_1+y_1)\sum\limits_{k=0}^{\infty} \sum\limits_{n+p=k} t_1^p y_1^n\\
& =  & 1 + (t_1+y_1)\sum\limits_{p=0}^{\infty} t_1^p\sum\limits_{n=0}^{\infty} y_1^n \\
& = & 1+\dfrac{t_1+y_1}{(1-t_1)(1-y_1)}\\
& = & \dfrac{1+ t_1y_1}{(1-t_1)(1-y_1)}.
\end{eqnarray*}
It follows that
\[
H(E;t_1,y_1)=\dfrac{1+ t_1y_1}{(1-t_1)(1-y_1)}.
\]
\end{proof}

In a similar way, we can calculate $H(E;\mathrm{T}_k,\mathrm{Y}_l)$ for any $k,l\in\mathbb{N}$ using Definition \ref{hsfunctions} with Theorems \ref{Berel1} and \ref{teo KRG}. The results is the following.

\begin{proposition}\label{HSE}
Let $k,l\in\mathbb N$. Then
\[
H(E;\mathrm{T}_k,\mathrm{Y}_l)=\frac{1}{2}\left(1+\prod_{i=1}^k\prod_{j=1}^l\frac{(1+t_i)(1+y_j)}{(1-t_i)(1-y_j)}\right).
\]
\end{proposition}

\begin{proof}
By the definition of $(k,l)$-semistandard tableau, we have only two types of tableaux for $\lambda=(p, 1^{n-p})$:
\ytableausetup
{mathmode, boxsize=1em}
\[
\begin{ytableau}
T & &T &Y & & Y\\
T  \\
\ \\
T \\
Y \\
\ \\
Y
\end{ytableau} \hspace{1.5cm} \begin{ytableau}
Y & &Y & Y& & Y\\
Y  \\
\ \\
Y\ \\
\ \\
\ \\
Y
\end{ytableau}
\]
where with the symbol $T$, we mean  ``elements lying in $\mathrm{T}_k$'' and with the symbol $Y$ ``elements lying in $\mathrm{Y}_l$''. Remember that
\begin{itemize}
\item The elements in $\mathrm{T}_k$ are nondecreasing in rows and  strictly increasing in columns.
\item The elements in $\mathrm{Y}_l$ are nondecreasing in column and  strictly increasing in rows.
\end{itemize}
Hence the tableaux of the first type have $n\geq 1$ boxes and contain at least one symbol $T$. The tableaux of the second type  do not contain symbols $T$ and have $n\geq 0$ boxes.

Consider a tableau $T_{\lambda}$ of the first type. The $T$-parts of $T_{\lambda}$ forms a semistandard $T_{\mu}$ filled with elements from $\mathrm{T}_k$ where $\mu=(q,q^{m-q})$ for some $m\leq n$ and $q\leq p$.
Hence the $T$-parts of such tableaux are in one-to-one correspondence with the semistandard $\mu$-tableaux filled with elements from $T_k$ where $\mu=(q,q^{m-q})$, $m\leq n$ and $q\leq p$.
The sum on all $\mu$ of the products of the entries of $T_{\mu}$ is equal to the sum of the Schur functions $S_{\mu}(\mathrm{T}_k)$.
If the $Y$-part of the arm of the tableau $T_{\lambda}$ consists of $y_{j_1},\ldots, y_{j_r}$, then  $1\leq j_1\cdots <j_r\leq l$.
Similarly, if the $Y$-part of the leg of the tableau of $T_{\lambda}$ consists of $y_{m_1},\ldots, y_{m_s}$ then $1\leq m_1\cdots\leq m_{s}\leq l$.
Hence the sum of all monomials $w^{T_{\lambda}}$, when $T_{\lambda}$ runs on all $(k-l)$-semistandard tableaux of type $1$, is
\begin{eqnarray*}
\sum\limits_{m\geq 1}\sum\limits_{q=1}^m S_{(q, 1^{m-q})}(\mathrm{T}_k)\sum\limits_{c_i\geq 0} y_1^{c_1}\cdots y_l^{c_l}\sum\limits_{j_1<\cdots<j_s} y_{j_1}\cdots y_{j_s}
& = & \sum\limits_{m\geq 1}\sum\limits_{q=1}^m S_{(q, 1^{m-q})}(\mathrm{T}_k)\prod_{j=1}^l \dfrac{1}{1-y_j}\sum\limits_{s=0}^l e_s(\mathrm{Y}_l)\\
& = & \sum\limits_{m\geq 1}\sum\limits_{q=1}^m S_{(q, 1^{m-q})}(\mathrm{T}_k)\prod_{j=1}^l \dfrac{1+y_j}{1-y_j}
\end{eqnarray*}
where
\[
e_s(\mathrm{Y}_l)= \sum\limits_{j_1<\cdots<j_s} y_{j_1}\cdots y_{j_s}
\]
is the $s$-th elementary symmetric function. By Theorem \ref{teo KRG}
\[
\sum\limits_{m\geq 1}\sum\limits_{q=1}^m S_{(q, 1^{m-q})}=H(E;\mathrm{T}_k)-1.
\]
The explicit form of $H(E;\mathrm{T}_k)$ is well known (see \cite{D2}). In particular, we have
\[
H(E;\mathrm{T}_k)=\dfrac{1}{2}\left(1+\prod_{i=1}^k \dfrac{1+t_i}{1-t_i}\right).
\]
In this way, the sum of all monomials $w^{T_{\lambda}}$, when $T_{\lambda}$ runs all $(k,l)$-semistandard tableaux of type $1$, has the form
\[
\dfrac{1}{2}\left(-1+\prod_{i=1}^k \dfrac{1+t_i}{1-t_i}\right)\prod_{j=1}^l \dfrac{1+y_j}{1-y_j}.
\]\\

Now, we consider the $(k,l)$-semistandard tableaux $T_{(p,1^{n-p})}$ of the second type. Clearly, the transposed tableau $T'_{(p,1^{n-p})}$ is the tableau of shape $(n-p+1, 1^{p-1})$.
So the entries of $T'_{(p,1^{n-p})}$  do not decrease in the first row and strictly increase in the columns. Hence $T'_{(p,1^{n-p})}$ is a semistandard tableau in the ordinary sense.
Applying the same arguments as for the $T$-part of the sum  for the tableaux of the first type, we have that the sum of the monomials $w^{T_{\lambda}}$ on all $T^{(p, 1^{n-p})}$ is equal to
\[
1+\sum_{n\geq 1}\sum_{p=1}^n S_{(n-p+1, 1^{p-1})}(Y_{l})=H(E;\mathrm{Y}_l)= \dfrac{1}{2}\left(1+ \prod_{j=1}^l \dfrac{1+y_j}{1-y_j}\right).
\]
Hence
\begin{eqnarray*}
H(E; \mathrm{T}_k,\mathrm{Y}_l) & = & \dfrac{1}{2}\left(-1+\prod_{i=1}^k \dfrac{1+t_i}{1-t_i}\right)\prod_{j=1}^l \dfrac{1+y_j}{1-y_j}+ \dfrac{1}{2}\left(1+ \prod_{j=1}^l \dfrac{1+y_j}{1-y_j}\right)\\
& = & \dfrac{1}{2}\left(-\prod_{j=1}^l\dfrac{1+y_j}{1-y_j} +\prod_{i=1}^k\prod_{j=1}^l \dfrac{(1+t_i)}{(1-t_i)} \dfrac{(1+y_j)}{(1-y_j)} + 1+ \prod_{j=1}^l \dfrac{1+y_j}{1-y_j}\right)\\
& = & \dfrac{1}{2}\left(1+ \prod_{i=1}^k\prod_{j=1}^l \dfrac{(1+t_i)}{(1-t_i)} \dfrac{(1+y_j)}{(1-y_j)}\right)
\end{eqnarray*}
and the proof follows.
\end{proof}

\begin{corollary}\label{HSUE}
Let $E$ be the infinite dimensional Grassmann algebra. Consider the algebra $UT_n(E)$ of $n\times n$ upper triangular matrices with entries in  $E$. Then for $k,l\in\mathbb{N}$
\[
H(UT_n(E);\mathrm{T}_k, \mathrm{Y}_l)= \sum\limits_{j=1}^n \binom{n}{j}\left( \dfrac{1}{2}\left[1+ \prod_{i=1}^k\prod_{s=1}^l \dfrac{(1+t_i)}{(1-t_i)} \dfrac{(1+y_s)}{(1-y_s)}\right]\right)^j\left(\sum\limits_{i=1}^{k}t_i+\sum\limits_{s=1}^{l}y_s-1          \right)^{j-1}.
\]
\end{corollary}

\begin{proof}
By  Proposition \ref{HSE} and Corollary \ref{prodhs}, we have
\[
H(UT_n(E);\mathrm{T}_k, \mathrm{Y}_l)= \sum_{j=1}^{n}\binom{n}{j}\left( \dfrac{1}{2}\left[1+ \prod_{i=1}^k\prod_{s=1}^l \dfrac{(1+t_i)}{(1-t_i)} \dfrac{(1+y_s)}{(1-y_s)}\right]\right)^j{(S_{(1)}(\mathrm{T}_k;\mathrm{Y}_l)-1)}^{j-1}.
\]
Note that by the definition of hook Schur functions, it follows that
\[
S_{(1)}(\mathrm{T}_k;\mathrm{Y}_l)= \sum\limits_{i=1}^{k}t_i+\sum\limits_{s=1}^{l}y_s.
\]

\noindent
Hence
\[
H(UT_n(E);\mathrm{T}_k, \mathrm{Y}_l)= \sum\limits_{j=1}^n \binom{n}{j}\left( \dfrac{1}{2}\left[1+ \prod_{i=1}^k\prod_{s=1}^l \dfrac{(1+t_i)}{(1-t_i)} \dfrac{(1+y_s)}{(1-y_s)}\right]\right)^j\left(\sum\limits_{i=1}^{k}t_i+\sum\limits_{s=1}^{l}y_s-1          \right)^{j-1}.
\]
\end{proof}

By Theorem \ref{existegancho}, there are natural numbers $k$ and $l$ such that $\chi(UT_n(E))\subseteq H(k,l)$. By Theorem \ref{mRk} it follows that $k=n$ and $l=2n-1$.
So, we have the following result.

\begin{proposition}\label{ganchoUE}
Let $n\geq 1$ and consider the algebra $UT_n(E)$. Then the partitions $\lambda$ with nonzero multiplicities $m_{\lambda}(UT_n(E))$ in the cocharacter sequence of $UT_n(E)$ are in the hook $H(n,2n-1)$.
\end{proposition}

Using the double Hilbert series of $UT_n(E)$, we are able to give a better description of the partitions $\lambda$  with $m_{\lambda}(UT_n(E))\neq 0$ than the description given in Theorem \ref{mRk}.

\begin{theorem}\label{ganchoE}
The hook Schur functions $HS_\pi(\mathrm{T}_k,\mathrm{Y}_l)$ participating in the product $H(E;\mathrm{T}_k,\mathrm{Y}_l)^j$ are indexed by partitions $\pi$ lying in $H(j,j)$.
\end{theorem}

\begin{proof}
We shall prove the assertion by induction on $j$. For $j=1$ the result follows from Theorems \ref{Berel1} and \ref{teo KRG}.

Assuming that the result is true for $j-1\geq 1$, we shall prove it for $j$. Note that
\[
H(E;\mathrm{T}_k,\mathrm{Y}_l)^j= H(E;\mathrm{T}_k,\mathrm{Y}_l)^{j-1}H(E;\mathrm{T}_k,\mathrm{Y}_l).
\]
By the induction hypotheses, we have
\[
H(E;\mathrm{T}_k,\mathrm{Y}_l)^{j-1}=\sum\alpha_{\lambda}HS_{\lambda}(\mathrm{T}_k,\mathrm{Y}_l),
\]
where $\lambda\in H(j-1,j-1)$ and $\alpha_{\lambda}\in\mathbb{N}_0=\mathbb{N}\cup\{0\}$.

Let $\lambda\in H(j-1,j-1)$ and consider the product
\[
HS_{\lambda}(\mathrm{T}_k,\mathrm{Y}_l)HS_{(q, 1^{m-q})} (\mathrm{T}_k,\mathrm{Y}_l)=\sum\limits \theta_{\nu} HS_{\nu}(\mathrm{T}_k,\mathrm{Y}_l).
\]
Suppose that there exists a partition $\nu$ such that $HS_{\nu}(\mathrm{T}_k,\mathrm{Y}_l)$ participates in the decomposition of the above product and $\nu\notin H(j,j)$.

Note that $\nu\notin H(j,j)$ implies $\nu_{j+1}\geq j+1$.
Let $Q\subseteq\nu/\lambda $ be the square
formed by the boxes $(j,j), (j,j+1), (j+1,j)$ and $(j+1,j+1)$ as in the picture below:
\ytableausetup
{mathmode, boxsize=1.5em}
\[
Q=\begin{ytableau}
\ & \ \\
\ & \  \\
\end{ytableau}.
\]
Of course, if we consider a semistandard tableau of shape ${\nu/\lambda}$ and content $(q, 1^{m-q})$, then $Q$ must be filled in as below
\begin{enumerate}
\item[(i)]
\ytableausetup
{mathmode, boxsize=1.5em}
\[
\begin{ytableau}
1 & 1 \\
a & b \\
\end{ytableau},
\]
where $a,b\neq 1$ and $a<b$.
\item[(ii)]
\ytableausetup
{mathmode, boxsize=1.5em}
\[
\begin{ytableau}
1 & a \\
b & c \\
\end{ytableau},
\]
where $a,b,c\neq 1$, $a<c$ and $b<c$.
\item[(iii)]
\ytableausetup
{mathmode, boxsize=1.5em}
\[
\begin{ytableau}
a & b \\
c & d \\
\end{ytableau},
\]
where $a,b,c,d\neq 1$, $a<b<d$ and $a<c<d$. Observe that those conditions are because we are working with content $(q,1^{m-q})$.
\end{enumerate}
Note that $a,b,c,d$ must be  pairwise different. Since  we consider a semistandard tableau of shape ${\nu/\lambda}$ and content $(q, 1^{m-q})$,
then none of the cases above yields lattice permutations when we read their entries from the right to the left and downwards.
Hence by the Littlewood-Richardson rule, $HS_{\nu}(\mathrm{T}_k,\mathrm{Y}_l)$ cannot participate in the decomposition of
\[
HS_{\lambda}(\mathrm{T}_k,\mathrm{Y}_l)HS_{(q, 1^{m-q})} (\mathrm{T}_k,\mathrm{Y}_l),
\]
that is an absurd and we are done.
\end{proof}

In what follows we would like to find the partitions with nonzero multiplicities participating in the decomposition of $H(UT_n(E);\mathrm{T}_k,\mathrm{Y}_l)$.

Given $n\geq 1$, let $Q_2=(n-1)^{n-1}$ be a square Young diagram with $n-1$ rows. We denote by $H(n,n)*Q_2$ the skew hook obtained
by identifying the box $(1,1)$ of $Q_2$  with the $(n+1,n+1)$-box of $H(n,n)$, the box $(1,2)$ of $Q_2$  with the $(n+1,n+2)$-box of $H(n,n)$ and so on.
Finally we identify the box $(n,n)$ of $Q_2$ with the $(2n,2n)$-box of $H(n,n)*Q_2$. In other words $\lambda\in H(n,n)*Q_2$ if $\lambda_{n+1}\leq 2n$, $\lambda_{2n+1}\leq n$
and $Q_2$ becomes the $n\times n$ square with $(n+1,n+1)$ as an upper left box.

\begin{proposition}\label{propQ2}
If $m_{\lambda}(UT_{n}(E))\neq 0$, then $\lambda\in H(n,n)*Q_2$ and  $|\lambda\cap Q_2|\leq n-1$.
\end{proposition}

\begin{proof}
By  Corollary \ref{HSUE}, the nonzero multiplicities $m_{\lambda}(UT_n(E))$ in the cocharacter sequence of $UT_n(E)$ come from the decomposition as an infinite sum of hook Schur functions of
\[
H(E; \mathrm{T}_k, \mathrm{Y}_l)^j\left(\sum_{i=1}^kt_i+\sum_{s=1}^ly_s\right)^{q}=H(E; \mathrm{T}_k,\mathrm{Y}_l)^j HS_{(1)}(T_k,Y_l)^{q},
\]
where $j\leq n$ and $q\leq n-1$, for some $k,l$. By Theorem \ref{ganchoE}, the hook Schur functions $HS_{\pi}(\mathrm{T}_k,\mathrm{Y}_l)$
participating in the product $H(E; \mathrm{T}_k, \mathrm{Y}_l)^j$ are indexed by partitions $\pi$ lying in $H(n,n)$, and hence $\pi_{n+1}\leq n$.
Note that by the Branching rule, the product $HS_{\pi}(\mathrm{T}_k,\mathrm{Y}_l)\cdot HS_{(1)}(T_k,Y_l)$ gives a sum of $HS_{\rho}(T,U)$
where the diagrams of $\rho$ are obtained from the diagram of $\pi$ by adding a box.
It follows that the diagram of $\rho$ has no more than one box in $Q_2$.
Multiplying $q$ times by $HS_{(1)}(\mathrm{T}_k,\mathrm{Y}_l)$ we add to the diagram of $\pi$ not more than $q\leq n-1$ boxes in $Q_2$.
This means that if $m_{\lambda}(E)\neq 0$, then $|\lambda\cap Q_2|\leq n-1$.

If there is $\lambda\notin H(n,n)*Q_2$ such that $m_{\lambda}(E)\neq 0$, then $\lambda_{n+1}\geq 2n$ or $\lambda_{2n}\geq n+1$.

Suppose $\lambda_{n+1}\geq 2n$. We know that there is $\pi\in H(n,n)$ such that the diagram $D_{\lambda}$ is obtained from the diagram $D_{\pi}$
by adding $q\leq n-1$ boxes. So $\lambda_{n+1}\leq 2n-1$, which is an absurd.

If $\lambda_{2n}\leq n+1$, we know that there is $\pi\in H(n,n)$ such that the diagram $D_{\lambda}$ is obtained from the diagram $D_{\pi}$ by adding $q\leq n-1$ boxes.
Note that the limit of cases is $\pi_i=n$ for $i\geq n+1$. Hence $\lambda_{2n}\leq n$, which is a contradiction and we are done.
\end{proof}

\begin{remark}
Note that the previous proposition gives a better description of the partitions $\lambda=(\lambda_1,\ldots,\lambda_p)$ such that $m_{\lambda}(UT_n(E))\neq 0$ than Theorem \ref{mRk}.
In fact, by Proposition \ref{propQ2}, we have that $\lambda\in H(n,2n-1)$, $\lambda_i\leq 2n-1$ for $n+1\leq i\leq 2n-1$ and $\lambda_i\leq n$ for $i\geq 2n$.
On the other hand Theorem \ref{mRk} gives only that $\lambda\in H(n,2n-1)$, that means $\lambda_i\leq 2n-1$ for $i\geq n+1$.
\end{remark}

\section{The $(k,l)$-multiplicity series of $UT_n(E)$}
\subsection{The $(k,l)$-multiplicity series}
Recall that if $A$ is a finite dimensional PI-algebra, it suffices to work with a sufficiently big set of variables $\mathrm{T}$
in order to capture all the multiplicities $m_\lambda(A)$ of its cocharacter sequence from its Hilbert series.
But if $A$ is an infinite dimensional algebra, its multiplicity series may be not enough to find all multiplicities $m_{\lambda}(A)$.
Due to this fact, we want to generalize the concept of the multiplicity series defining the $(k,l)$-multiplicity series of $A$.
This series contains the information about the multiplicities $m_{\lambda}(A)$ for $\lambda$ in the hook $H(k,l)$.


As in \cite{BR3}, identifying a partition with its Young diagram, we can break each $\lambda\in H(k,l)$ into three parts $\lambda\rightarrow(\lambda_0,\mu,\nu)$
where $\lambda_0$ is the piece of the partition in the $k\times l$ rectangle $(l^k)$; $\mu$ is a partition with at most $k$ parts and it is the part of $\lambda$ to the right of $\lambda_0$;
$\nu$ is a partition with at most $l$ parts such that its conjugate is the part of $\lambda$ below $\lambda_0$, see Figure \ref{defuv}.
\begin{figure}[!htb]
\centering
\begin{tikzpicture}
\draw (0,0.5)--(7,0.5);
\draw[dashed] (7,0.5)--(8.5,0.5);
\draw (0, 0.5)--(0,-5);
\draw [dashed] (0,-5)--(0,-6.5);
\draw (1,-5.8) node[above] {$l$};
\draw (5,0.5)--(5,-0.2);
\draw (4.0,-0.2)--(5,-0.2);
\draw (4,-0.2)--(4,-0.6);
\draw (2,-0.6)--(4,-0.6);
\draw (2,-0.6)--(2,-1);
\draw (1.5,-1)--(2,-1);
\draw (1.5,-1)--(1.5,-2.5);
\draw (1.1,-2.5)--(1.5,-2.5);
\draw (1.1, -2.5)--(1.1,-3);
\draw (0.5, -3)--(1.1,-3);
\draw (0.5,-3)--(0.5,-4);
\draw (0,-4)--(0.5,-4);

\draw (0.3,-0.5) node[right] {{$\lambda_{0}$}};
\draw [<->] (7,0.5)--(7,-1.5);
\draw (2.5,-0.3) node[right] {{$\mu$}};
\draw (7,-0.3) node[right] {{$k$}};
\draw [dashed](7,-1.5)--(8.5,-1.5);
\draw (2.0,-1.5)--(7,-1.5);
\draw [dashed][red](2,0.5)--(2,-1.5);
\draw [dashed][red](0,-1.5)--(2,-1.5);
\draw [<->](0,-5)--(2.0,-5);
\draw (2,-1.5)--(2.0,-5);
\draw[dashed](2.0,-5)--(2.0,-6.5);
\draw (0.5,-2.5) node[right] {{$\nu'$}};


\end{tikzpicture}
\caption{Definition of $\lambda_{0}, \mu$ and $\nu$}
\label{defuv}
\end{figure}
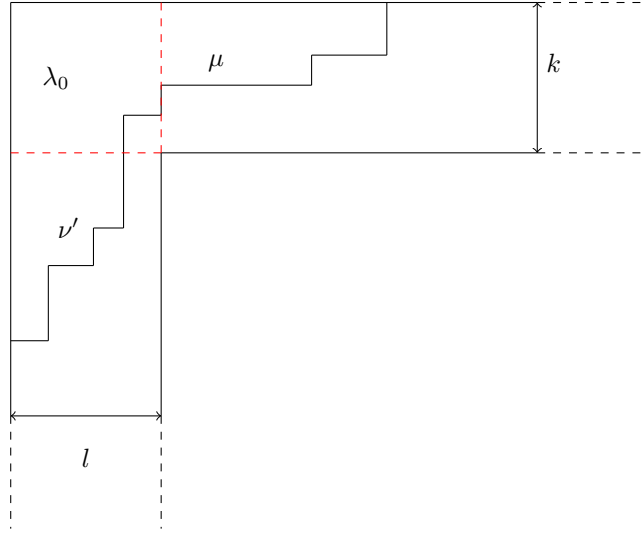
\FloatBarrier
We fix two nonnegative integers $k$ and $l$ such that $k+l\geq 1$. Let $\lambda_{0}$ be a partition such that $\lambda_{0}\subseteq (l^k)$ and let
\[
H_{\lambda_{0}}(k,l):=\{\lambda\in H(k,l)\mid \lambda\cap (l^k)=\lambda_{0}\}.
\]
Notice that $H(k,l)=\displaystyle\bigcup\limits_{\lambda_{0}\subseteq (l^k)} H_{\lambda_{0}}(k,l)$.
Let $\mathrm{T}_k=\{t_1,\ldots, t_k\}$, $Y=\{y_1,\ldots,y_l\}$ and $V=\{v_1,\ldots, v_k \}$ be three sets of commutative variables
and consider the algebra
\[
\mathbb{C}[[\mathrm{T}_k,\mathrm{Y}_l]]=\mathbb{C}[[t_1,\ldots,t_k, y_1,\ldots, y_l]]
\]
of formal power series in $(k+l)$ commutative variables. Let
\[
\Lambda^{(k,l)}:=\{\sum\limits_{\lambda} m_{\lambda}HS_{\lambda}(\mathrm{T}_k, \mathrm{Y}_{l})\mid \lambda\in H(k,l),\quad m_{\lambda}\in\mathbb{C}\}.
\]
\[
\Lambda^{(k,l,n)}:=\{\sum\limits_{\lambda} m_{\lambda}HS_{\lambda}(\mathrm{T}_k, \mathrm{Y}_{l})\mid \lambda\in H(k,l;n),\quad m_{\lambda}\in\mathbb{C}\}.
\]
Note that $\Lambda^{(k,l)}$ is a subalgebra of $\mathbb{C}[[\mathrm{T}_k,\mathrm{Y}_{l}]]$ because the hook Schur functions are multiplied with the Littlewood-Richardson rule.
The set $\{HS_{\lambda}(\mathrm{T}_k, \mathrm{Y}_{l})\mid \lambda\in H(k,l;n)\}$ is a basis of  $\Lambda^{(k,l,n)}$ as a vector space (see \cite{BR2}).

Given $g(\mathrm{T}_k,\mathrm{Y}_l)=\sum\limits_{\lambda\in H(k,l)}m_{\lambda}HS_{\lambda}(\mathrm{T}_k,\mathrm{Y}_l)\in\Lambda^{(k,l)}$, we have
\[
g(\mathrm{T}_k,\mathrm{Y}_l)=\sum\limits_{\lambda_{0}\subseteq(l^k)}\sum\limits_{\lambda\in H_{\lambda_0}(k,l)} m_{\lambda}HS_{\lambda}(\mathrm{T}_k,\mathrm{Y}_l).
\]

\begin{definition}
For
\[
g(\mathrm{T}_k,\mathrm{Y}_l)=\sum\limits_{\lambda_{0}\subseteq(l^k)}\sum\limits_{\lambda\in H_{\lambda_0}(k,l)} m_{\lambda}HS_{\lambda}(\mathrm{T}_k,\mathrm{Y}_l)\in \Lambda^{(k,l)}
\]
we define the $(k,l)$-\textit{multiplicity series} of $g$ by
\[
\widehat{M}(g;\mathrm{V}_k,\mathrm{T}_k,\mathrm{Y}_l):=\sum\limits_{\lambda_{0}\subseteq(l^k)}\sum\limits_{\lambda\in H_{\lambda_0}(k,l)}m_{\lambda}\mathrm{V}_k^{\lambda_0}\mathrm{T}_k^{\mu}\mathrm{Y}_l^{\nu},
\]
where $\mathrm{V}_k^{\lambda_0}=v_1^{{\lambda_0}_1}\cdots v_k^{{\lambda_0}_k}$, $\mathrm{T}_k^{\mu}= t_1^{\mu_1}\cdots t_k^{\mu_k}$ and $\mathrm{Y}_l^{\nu}= y_1^{\nu_1}\cdots y_l^{\nu_l}$.
\end{definition}

It is clear that $\widehat{M}(g;\mathrm{V}_k,\mathrm{T}_k,\mathrm{Y}_l)$ is an element of $\mathbb{C}[[\mathrm{V}_k,\mathrm{T}_k,\mathrm{Y}_l]]$,
the algebra of formal power series in $(2k+l)$ variables. Observe that $\widehat{M}$ defines an injective linear map from $\Lambda^{(k,l)}$ to $\mathbb{C}[[\mathrm{V}_k, \mathrm{T}_k, \mathrm{Y}_l]]$
and set $\Lambda_{(2k,l)}:=\widehat{M}(\Lambda^{(k,l)})$.

\begin{example}
Consider the hook $H(2,1)$ and the partition $\lambda=(2, 1^2)$. Then $\lambda_0=(1,1),\mu=(1)$ and $\nu=(1)$. It follows that
\[
\widehat{M}(HS_{\lambda}(\mathrm{T}_2, \mathrm{Y}_1))=v_1v_2t_1y_1.
\]
Now, if we consider the hook $H(3,1)$ and the same partition, we obtain $\lambda_0=(1^3)$, $\mu=(1)$ and $\nu=(0)$. Hence
\[
\widehat{M}(HS_{\lambda}(\mathrm{T}_3, \mathrm{Y}_1))=v_1v_2v_3t_1.
\]
\end{example}

\begin{definition}
Let $A$ be a PI-algebra.
The formal series
\[
\widehat{M}(A;\mathrm{V}_k,\mathrm{T}_k,\mathrm{Y}_l)=\sum\limits_{\lambda_0\subseteq (l^k)}\sum\limits_{\lambda\in H_{(\lambda_0)}(k,l)}m_\lambda(A) \mathrm{V}_k^{\lambda_0}\mathrm{T}_k^{\mu}\mathrm{Y}_{l}^{\nu}
\]
where $\lambda\in H(k,l)$ and $m_{\lambda}(A)$ is the multiplicity corresponding to $\chi_{\lambda}$ in the cocharacter sequences of $A$, is called the \textit{ $(k,l)$-multiplicity series of $A$}.
\end{definition}

When the sets of variables are well known, we may also write $\widehat{M}(A)$ instead of $\widehat{M}(A;\mathrm{V}_k,\mathrm{T}_k,\mathrm{Y}_l)$.\\

Our next step is to find an expression for the $(k,l)$-multiplicity series of $UT_n(E)$. At light of Proposition \ref{HSE}, we define the linear operator
\[
G:\Lambda_{(2k,l)}\rightarrow  \Lambda_{(2k,l)}
\]
such that
\[
G(\widehat{M}(g))=\widehat{M}\left(g\cdot\frac{1}{2}\left(1+\prod_{i=1}^k\prod_{s=1}^l\frac{(1+t_i)(1+y_s)}{(1-t_i)(1-y_s)}\right)\right),
\]
where $g\in\Lambda^{(k,l)}$.

\begin{remark}\label{HSn}
Notice that  $HS_{(q, 1^{m-q})}(\mathrm{T}_k,\mathrm{Y}_l)$ participates in the decomposition of
$HS_{(q)}(\mathrm{T}_k,\mathrm{Y}_l)HS_{(1)}(\mathrm{T}_k,\mathrm{Y}_l)$ and
$HS_{(q-1)}(\mathrm{T}_k,\mathrm{Y}_l)HS_{1^{(m-q+1)}}(\mathrm{T}_k,\mathrm{Y}_l)$ as a sum of hook Schur functions. Hence $HS_{(q, 1^{m-q})}$ appears with multiplicity $2$ in the product
\[
\sum\limits_{n\geq 0} HS_{(n)}(\mathrm{T}_k,\mathrm{Y}_l)\sum\limits_{m\geq 0} HS_{(1^{m})}(\mathrm{T}_k,\mathrm{Y}_l).
\]
It follows that
\[
H(E;\mathrm{T}_k,\mathrm{Y}_l)=1+\sum\limits_{m\geq 1}\sum_{q=1}^m HS_{(q,1^{m-q})}(\mathrm{T}_k,\mathrm{Y}_l)=\dfrac{1}{2}\left( 1+\sum\limits_{n\geq 0} HS_{(n)}(\mathrm{T}_k,\mathrm{Y}_l)\sum\limits_{m\geq 0} HS_{(1^{m})}(\mathrm{T}_k,\mathrm{Y}_l)\right).
\]
It is well known that
\[
\sum\limits_{n\geq 0} HS_{(n)}(\mathrm{T}_k,\mathrm{Y}_l)= \prod\limits_{i=1}^k \prod\limits_{s=1}^l\dfrac{1+y_s}{1-t_i},
\]
\[
\sum\limits_{m\geq 0} HS_{(1^m)}(\mathrm{T}_k,\mathrm{Y}_l)=\prod\limits_{i=1}^k \prod\limits_{s=1}^l\dfrac{1+t_i}{1-y_s}.
\]
\end{remark}

Due to Remark \ref{HSn}, we define the following two operators $G_1:\Lambda_{(2k,l)}\rightarrow  \Lambda_{(2k,l)}$ and $G_2:\Lambda_{(2k,l)}\rightarrow  \Lambda_{(2k,l)}$ given by
\[
G_1(\widehat{M}(g))=\widehat{M}\left(g\cdot \prod_{i=1}^k\prod_{s=1}^l\frac{1+y_s}{1-t_i}\right)
\]
and
\[
G_2(\widehat{M}(g))=\widehat{M}\left(g\cdot \prod\limits_{i=1}^k \prod\limits_{s=1}^l\dfrac{1+t_i}{1-y_s} \right),
\]
where $g\in\Lambda^{(k,l)}$.

\noindent
Note that $G_1\circ G_2=G_2\circ G_1$ and $G=\dfrac{1}{2}\left(\mathbb{1}+G_2\circ G_1\right)$ where $\mathbb{1}$ is the identity map.

\subsection{The action of $G_1$ and $G_2$}
Now, we are going to describe the action of $G_1$ and $G_2$ on $\Lambda_{2k,l}$. Let us start with the operator $G_1$.
Using the notation of Section $3$, we define the following linear operator.
\begin{definition}
Given a positive integer $d$. Let $f(\mathrm{T}_d)\in \mathbb{C}[[\mathrm{T}_d]]^{S_d}$, define the conjugate Young operator $\overline{Y}$ on  $\mathbb{C}[[\mathrm{V}_d]]$ as 
$$\overline{Y}(M(f(\mathrm{T}_d)):=M\left( f(\mathrm{T}_d)\cdot \sum\limits_ {s=0}^d S_{(1^s)}(\mathrm{T}_d)\right).$$
\end{definition}
\begin{lemma}\label{iyoung1}
Consider the hook $H(k,l)$ and let $\lambda\in H(k,l)$ be a partition such that $\lambda_0=(l^k)$. Then
{\small
\begin{eqnarray}\label{igualdade young}
\widehat{M}\left(HS_{\lambda}(\mathrm{T}_k,\mathrm{Y}_l)HS_{(n)}(\mathrm{T}_k,\mathrm{Y}_l)\right)=\mathrm{V}_k^{\lambda_0}\sum\limits_{m=0}^n M(S_{\mu}(\mathrm{T}_k)S_{(m)}(\mathrm{T}_k)) M(S_{\nu}(\mathrm{Y}_l)S_{(1^{n-m})}(\mathrm{Y}_l))
\end{eqnarray}}
\end{lemma}
\begin{proof}
Recall that $HS_{\beta}(\mathrm{T}_k,\mathrm{Y}_l)=0$ if, and only if, $\beta\notin H(k,l)$.
Since $\lambda_0=(l^k)$, if $HS_{\beta}(\mathrm{T}_k,\mathrm{Y}_l)$ participates in the decomposition of $$HS_{\lambda}(\mathrm{T}_k,\mathrm{Y}_l)HS_{(n)}(\mathrm{T}_k,\mathrm{Y}_l),$$
then $\beta_0=\beta\cap (l^k)=(l^k)=\lambda_0$. Hence, applying the Young rule to the partition $\lambda$ is equivalent to applying the Young rule to $\mu$ and $\nu$. So, we have
\begin{eqnarray*}
\widehat{M}\left(HS_{\lambda}(\mathrm{T}_k,\mathrm{Y}_l)HS_{(n)}(\mathrm{T}_k,\mathrm{Y}_l)\right) & = & v_1^l\cdots v_k^l\sum\limits_{m=0}^n M(S_{\mu}(\mathrm{T}_k)S_{(m)}(\mathrm{T}_k)) M(S_{\nu}(\mathrm{Y}_l)S_{(1^{n-m})}(\mathrm{Y}_l))\\
& = & \mathrm{V}_k^{\lambda_0}\sum\limits_{m=0}^n M(S_{\mu}(\mathrm{T}_k)S_{(m)}(\mathrm{T}_k)) M(S_{\nu}(\mathrm{Y}_l)S_{(1^{n-m})}(\mathrm{Y}_l))
\end{eqnarray*}
\end{proof}
\begin{example}
\textnormal{Consider the hook $H(2,3)$ and the partition $\lambda=(3,3,2,1)\vdash 9$. Note that $\lambda\in H(2,3)$, $\lambda_0=(3,3), \mu=(0)$ and $\nu=(2,1)$. We want to calculate
\[
\widehat{M}\left(HS_{\lambda}(\mathrm{T}_2,\mathrm{Y}_3)\cdot HS_{(3)}(\mathrm{T}_2,\mathrm{Y}_3)\right).
\]}\\
\textnormal{First, we shall calculate the $(2,3)$-multiplicity series using the Young rule for the tensor product of the partitions $\lambda$ and $(3)$. We get}
\begin{equation}\label{exem}
\begin{split}
HS_{\lambda}(\mathrm{T}_2,\mathrm{Y}_3)\cdot HS_{(3)}(\mathrm{T}_2,\mathrm{Y}_3)= & HS_{(6,3,2,1)}(\mathrm{T}_2,\mathrm{Y}_3) + HS_{(5,3,3,1)}(\mathrm{T}_2,\mathrm{Y}_3) \\
& +HS_{(5,3,2,2)}(\mathrm{T}_2,\mathrm{Y}_3) + HS_{(5,3,2,1,1)}(\mathrm{T}_2,\mathrm{Y}_3)+ HS_{(4,3,3,2)}(\mathrm{T}_2,\mathrm{Y}_3)\\
& +HS_{(4,3,3,3,1)}(\mathrm{T}_2,\mathrm{Y}_3) + HS_{(4,3,2,2,1)}(\mathrm{T}_2,\mathrm{Y}_3)+
 HS_{(3,3,3,2,1)}(\mathrm{T}_2,\mathrm{Y}_3).
\end{split}
\end{equation}
Then
\begin{eqnarray*}
\widehat{M}\left(HS_{\lambda}(\mathrm{T}_2,\mathrm{Y}_3)\cdot HS_{(3)}(\mathrm{T}_2,\mathrm{Y}_3)\right)& = & v_1^3v_2^3(y_1^3y_2 ^2y_3+  t_1y_1^3y_2^2+t_1y_1^2y_2^2y_3\\
 & & +t_1y_1^3y_1y_3
+t_1^ 2y_1^3y_2 + t_1^2y_1^2y_2^2+ t_1^2y_1^2y_2y_3+ t_1^3y_1^2y_2).
\end{eqnarray*}
\textnormal{Now, we are going to calculate $\widehat{M}\left(HS_{\lambda}(\mathrm{T}_2,\mathrm{Y}_3)\cdot HS_{(3)}(\mathrm{T}_2,\mathrm{Y}_3)\right)$
using the algorithm presented at the beginning of the section and the equality (\ref{igualdade young}). Notice that}
\begin{itemize}
\item  $S_{\mu}(\mathrm{T}_2)S_{(0)}(\mathrm{T}_2)=1$, $S_{\nu}(\mathrm{Y}_3)S_{(1^3)}(\mathrm{Y}_3)=S_{(3,2,1)}(\mathrm{Y}_3)$;
\item  $S_{\mu}(\mathrm{T}_2)S_{(1)}(\mathrm{T}_2)=S_{(1)}(\mathrm{T}_2)$, $S_{\nu}(\mathrm{Y}_3)S_{(1^2)}(\mathrm{Y}_3)=S_{(3,2)}(\mathrm{Y}_3)+ S_{(2,2,1)}(\mathrm{Y}_3)+ S_{(3,1,1)}(\mathrm{Y}_3)$;
\item $S_{\mu}(\mathrm{T}_2)S_{(2)}(\mathrm{T}_2)=S_{(2)}(\mathrm{T}_2)$, $S_{\nu}(\mathrm{Y}_3)S_{(1)}(\mathrm{Y}_3)=S_{(3,1)}(\mathrm{Y}_3)+ S_{(2,2)}(\mathrm{Y}_3)+ S_{(2,1,1)}(\mathrm{Y}_3)$;
\item $S_{\mu}(\mathrm{T}_2)S_{(3)}(\mathrm{T}_2)=S_{(3)}(\mathrm{T_2})$, $S_{\nu}(\mathrm{Y}_3)S_{(0)}(\mathrm{Y_3})=S_{\nu}(\mathrm{Y}_3)$.
\end{itemize}
\textnormal{Hence}
\begin{eqnarray*}
\widehat{M}\left(HS_{\lambda}(\mathrm{T}_2,\mathrm{Y}_3)\cdot HS_{(3)}(\mathrm{T}_2,\mathrm{Y}_3)\right)&  
 = & v_1^3v_2^3(y_1^3y_2 ^2y_3+  t_1y_1^3y_2^2+t_1y_1^2y_2^2y_3\\
 & & +t_1y_1^3y_1y_3
+t_1^ 2y_1^3y_2 + t_1^2y_1^2y_2^2+ t_1^2y_1^2y_2y_3+ t_1^3y_1^2y_2),\\
\end{eqnarray*}
\end{example}
\begin{lemma}
Let $\lambda$ be a partition in the hook $H(k,l)$ such that $\lambda_0\neq (l^k)$, then
{\footnotesize
\begin{eqnarray}
\widehat{M}\left(HS_{\lambda}(\mathrm{T}_k,\mathrm{Y}_l)HS_{(n)}(\mathrm{T}_k,\mathrm{Y}_l)\right)  =\sum\limits_{\beta_{0}\in\Pi} V_k^{\beta_{0}}\sum\limits_{p=0}^{n- |D_{\beta_0}\backslash D_{\lambda_0}|} M\left(S_{\mu} S_{(p)};\mathrm{T_{r_{\beta_{0}}}}\right)M\left(S_{\nu}S_{(1^{n-|D_{\beta_0}\backslash D_{\lambda_0}|-p})};\mathrm{Y}_c \right),
\end{eqnarray}}
where $\Pi=\{\beta_0\subseteq (l^k)| \beta \ \textrm{partipates in the decomposition of} \  HS_{\lambda}HS_{(n)}\}$, $r_{\beta_0}$ is the number of the rows of $\beta_0$ of size $l$ and $c$ is the number of the columns of $\lambda_0$ of size $k$.
\end{lemma}\label{iyoung2}
\begin{proof}
Note that the difference between this case and Lemma  \ref{iyoung1} is that if $\beta\in H(k,l)$ participates in the decomposition of 
 $$HS_{\lambda}(\mathrm{T}_k,\mathrm{Y}_l)HS_{(n)}(\mathrm{T}_k,\mathrm{Y}_l)$$
as sum of hook Schur functions, then $\beta_0$ is not necessarily $\lambda_0$.
Consider the set $\Pi$ and notice that if $\lambda_1\geq l$ or $\lambda_1'\geq k$, then $\lambda_0\in \Pi$. The possible $\beta_0$ are those whose diagrams are obtained from the diagram of $\lambda_0$ when we apply the Young rule to the partitions $\lambda_0$ and $(m)$  for some $0\leq m\leq n$ such that $D_{\beta_0}\subseteq D_{(l^k)}$.\\
\indent
Now, we identify $\lambda$ with the triple $(\lambda_0,\mu, \nu)$. Suppose that $D_{\beta_0}$ is obtained from $D_{\lambda_0}$ by adding $m$ boxes. Note that $m=|D_{\beta_0}\backslash D_{\lambda_0}|$, that is, $m$ is the number of boxes in the skew-diagram $D_{\beta_0}\backslash D_{\lambda_0}$.
If we want to know what partitions $\beta$  satisfy $\beta\cap (l^k)=\beta_0$, we have to add  a total of $n-m$ boxes to the diagrams $D_{\mu}$ and $D_{\nu}$ using the Young rule. Consider the numbers $r_{\beta_0}$ and $c$, observe that we can add boxes to the diagram $D_{\mu}$ up to line $r_{\beta_0}$.  In the case of the diagram $D_{\nu}$, it is only  allowed to add boxes up to line $c$.\\
Hence, the partitions $\beta\in H(k,l)$ participating in the decomposition of  
$$\widehat{M}(HS_{\lambda}(\mathrm{T}_k,\mathrm{Y}_l)HS_{(n)}(\mathrm{T}_k,\mathrm{Y}_l))$$
such that $\beta\cap (l^k)=\beta_0$ are determined by the following expression
$$ V_k^{\beta_{0}}\sum\limits_{p=0}^{n-|D_{\beta_0}\backslash D_{\lambda_0}|} M\left(S_{\mu} S_{(p)};\mathrm{T_{r_{\beta_{0,m}}}}\right)M\left(S_{\nu}S_{(1^{n-|D_{\beta_0}\backslash D_{\lambda_0}|-p})};\mathrm{Y}_c\right).$$
It follows that
{\small
\begin{equation*}\label{G1geral}
\widehat{M}\left(HS_{\lambda}(\mathrm{T}_k,\mathrm{Y}_l)HS_{(n)}(\mathrm{T}_k,\mathrm{Y}_l)\right)  =\sum\limits_{\beta_{0}\in\Pi} V_k^{\beta_{0}}\sum\limits_{p=0}^{n- |D_{\beta_0}\backslash D_{\lambda_0}|} M\left(S_{\mu} S_{(p)};\mathrm{T_{r_{\beta_{0}}}}\right)M\left(S_{\nu}S_{(1^{n-|D_{\beta_0}\backslash D_{\lambda_0}|-p})};\mathrm{Y}_c \right).
\end{equation*}
}
\end{proof}
Note that if $r_{\beta_0}=0$ or $c=0$, then $\emptyset=Y_{0}=T_{0}$. Hence $1=Y(1,\mathrm{T}_0)=\overline{Y}(1,\mathrm{Y}_0)$.
\begin{example}
\textnormal{Consider $\lambda=(5,1,1)\in H(2,2)$. In this case $\lambda_0=(2,1)$, $\mu=(3)$ and $\nu=(1)$.
If we want to calculate $\widehat{M}(HS_{\lambda}(\mathrm{T}_2,\mathrm{Y}_2)HS_{(2)}(\mathrm{T}_2,\mathrm{Y}_2))$, first note that if $HS_{\beta}(\mathrm{T}_2,\mathrm{Y}_2)$ participates in the decomposition of
\[
HS_{\lambda}(\mathrm{T}_2,\mathrm{Y}_2)HS_{(2)}(\mathrm{T}_2,\mathrm{Y}_2)
\]
as a sum of hook Schur functions then $\beta\in H(2,2)$, $\beta\vdash 9$ and $\beta_0=\lambda_0$ or $\beta_{0,1}=(2,2)$.
Note that the second possibility for $\beta_0$ is obtained by applying the Young rule for the partitions $\lambda_0$ and $(1)$}.

\noindent
\textnormal{For $\beta_0=(2,1)$, we have}
\begin{itemize}
\item $S_{\mu}(\mathrm{T}_1)S_{(2)}(\mathrm{T}_1)=S_{(5)}(\mathrm{T}_1)$, $S_{\nu}(\mathrm{Y}_1)S_{(0)}(\mathrm{Y}_1)=S_{\nu}(\mathrm{Y}_1)$,
\item $S_{\mu}(\mathrm{T}_1)S_{(1)}(\mathrm{T}_1)= S_{(4)}(\mathrm{T}_1)$, $S_{\nu}(\mathrm{Y}_1)S_{(1)}(\mathrm{Y}_1)= S_{(2)}(\mathrm{Y}_1)$,
\item  $S_{\mu}(\mathrm{T}_1)S_{(0)}(\mathrm{T}_1)=S_{\mu}(\mathrm{T}_1)$, $S_{\nu}(\mathrm{Y}_1)S_{(1^2)}(\mathrm{Y}_1)=0$.
\end{itemize}
\textnormal{For $\beta_0=(2,2)$, we have}
\begin{itemize}
\item $S_{\mu}(\mathrm{T}_2)S_{(1)}(\mathrm{T}_2)=S_{(4)}(\mathrm{T}_2)+ S_{(3,1)}(\mathrm{T}_2)$, $S_{\nu}(\mathrm{Y}_1)S_{(0)}(\mathrm{Y}_1)=S_{\nu}(\mathrm{Y}_1)$,
\item $S_{\mu}(\mathrm{T}_2)S_{(0)}(\mathrm{T}_2)=S_{\mu}(\mathrm{T}_2)$,
$S_{\nu}(\mathrm{Y}_1)S_{(1)}(\mathrm{Y}_1)=S_{(2)}(\mathrm{Y}_1)$.
\end{itemize}
\textnormal{Then}
\begin{eqnarray*}
\widehat{M}\left(S_{\lambda}(\mathrm{T}_2,\mathrm{Y}_2)\right) & = & v_1^2v_2( t_1^5y_1+ t_1^4y_1^2)+ v_1^2v_2^2(t_1^4y_1+t_1^3t_2y_1+ t_1^3y_1^2).
\end{eqnarray*}
\end{example}
Using Lemmas \ref{iyoung1}, \ref{iyoung2} and the linearity of $\widehat{M}$ and $M$, we have the following result

\begin{theorem}\label{G1des}
Let $\lambda$ be a partition in $H(k,l)$. The action of $G_1$ on $\widehat{M}(HS_{\lambda}(\mathrm{T}_k,\mathrm{Y}_l))$ can be described as follows
\begin{enumerate}
\item[{\rm (i)}] If $\lambda_0=(l^k)$, then
\[
G_1(\widehat{M}(HS_{\lambda}(\mathrm{T}_k,\mathrm{Y}_l)))=
 \mathrm{V}_k^{\lambda_0}Y(M(S_{\nu};T_k))\overline{Y}(M(S_{\mu};\mathrm{Y}_l)).
 \]
\item[{\rm (ii)}] If $\lambda_0\neq (l^k)$, then
\[
G_1(\widehat{M}(HS_{\lambda}(\mathrm{T}_k;\mathrm{Y}_l)))=\sum\limits_{\beta_0\in\Omega}\mathrm{V}_k^{\beta_0}Y(M(S_{\mu};\mathrm{T}_{r_{\beta_0}}))\overline{Y}(M(S_{\nu});\mathrm{Y}_c)),
\]
where $\Omega=\{\beta_0\subseteq (l^k)\mid D_{\beta_0}  \ \textrm{is obtained from}\ D_{\lambda_0 }\
\textrm{by the Young rule (Case 1)} \}$, $r_{\beta_0}$ is the number of rows of the diagram of $D_{\beta_0}$ of size $l$ and $c$ is the number of columns of the diagram of $D_{\lambda_0}$ of size $k$.
\end{enumerate}
\end{theorem}
Now, let us study $G_2$. Our goal is to describe $G_2$ in terms of the operators $Y$ and $\overline{Y}$ defined above.\\
Note that $(n)$ and $(1^n)$ are conjugate partitions. So, from Theorem \ref{G1des} we get the following result.

\begin{corollary}\label{G2des}
Let $\lambda$ be a partition in $H(k,l)$. The action of $G_2$ on $\widehat{M}(HS_{\lambda}(\mathrm{T}_k,\mathrm{Y}_l))$ can be described as follows:
\begin{enumerate}
\item[{\rm (i)}]  If $\lambda_0=(l^k)$ then
\[
G_2(\widehat{M}(HS_{\lambda}(\mathrm{T}_k,\mathrm{Y}_l)))=
 \mathrm{V}_k^{\lambda_0}\overline{Y}(M(S_{\nu};T_k))Y(M(S_{\mu};\mathrm{Y}_l))
 \]
\item[{\rm (ii)}] If $\lambda_0\neq (l^k)$ then
\[
G_2(\widehat{M}(HS_{\lambda}(\mathrm{T}_k,\mathrm{Y}_l)))=\sum\limits_{\beta_0\in\Omega'}\mathrm{V}_k^{\beta_0}\overline{Y}(M(S_{\mu};\mathrm{T}_{r}))Y(M(S_{\nu});\mathrm{Y}_{c_{\beta_0}})),
\]
where $\Omega'=\{\beta_0\subseteq (l^k)\mid D_{\beta_0}  \ \textrm{is obtained from}\ D_{\lambda_0 }\ \textrm{by Young rule (Case 2)} \}$, $r$ is the number of rows of the diagram of $D_{\lambda_0}$ of size $l$ and $c_{\beta_0}$ is the number of columns of the diagram of $D_{\beta_0}$ of size $k$.
\end{enumerate}
\end{corollary}

The following theorem is an analog of Corollary 13 of \cite{BD}. We obtain an expression for the $(k,l)$-multiplicity series of $UT_n(E)$.

\begin{theorem}\label{hookmultUTE}
Let $E$ be the infinite dimensional Grassmann algebra. Then
\[
\widehat{M}(UT_n(E);\mathrm{V}_k,\mathrm{T}_k,\mathrm{Y}_l)=\sum_{j=1}^n\sum_{q=0}^{j-1}\sum_{\lambda\vdash q}(-1)^{j-q-1}{n\choose j}{j-1\choose q}d_{\lambda}G^j(\mathrm{V}_k^{\lambda_0}\mathrm{T}_k^\mu \mathrm{Y}_l^\nu),
\]
where $d_{\lambda}$ is the degree of the $S_{\lambda}$-character $\chi_{\lambda}$ and
\[
T_k^\mu Y_l^\nu=t_1^{\mu_1}\cdots t_k^{\mu_k}y_1^{\nu_1}\cdots y_l^{\nu_l}
\]
for $\mu=(\mu_1,\ldots,\mu_k)$, $\nu=(\nu_1,\ldots,\nu_l)$ partitions outside the rectangle $(l^k)$.
\end{theorem}

\begin{proof}
Expanding the expression of $H(UT_n(E);\mathrm{T}_k,\mathrm{Y}_l)$ from Corollary \ref{HSUE} we obtain:
{\small
\begin{eqnarray*}
H(UT_n(E);\mathrm{T}_k,\mathrm{Y}_l) & = & \sum_{j=1}^{n} \binom{n}{j}{\left( \dfrac{1}{2}\left[1+ \prod_{i=1}^k\prod_{j=1}^l \dfrac{(1+t_i)}{(1-t_i)} \dfrac{(1+y_j)}{(1-y_j)}\right]\right)}^{j}\sum\limits_{q=0}^{j-1}{(-1)}^{j-1-q}\binom{j-1}{q}{\left(\sum\limits_{i=1} ^l t_i+ \sum\limits_{s=1} ^l y_s\right)}^{q}.
\end{eqnarray*}}
Since
\[
{\left(\sum\limits_{i=1} ^l t_i+ \sum\limits_{s=1} ^l y_s\right)}^{q}=HS_{(1)}(\mathrm{T}_k,\mathrm{Y}_l)^q=\sum\limits_{\lambda\vdash q}d_{\lambda}HS_{\lambda}(\mathrm{T}_k,\mathrm{Y}_l),
\]
where $d_{\lambda}$ is the degree of the $S_{|\lambda|}$-character $\chi_{\lambda}$, it follows that
{\small
\[
H(UT_n(E);\mathrm{T}k,\mathrm{Y}_l)=\sum_{j=1}^n\sum_{q=0}^{j-1}\sum_{\lambda\vdash q}(-1)^{j-q-1}{n\choose j}{j-1\choose q}d_{\lambda}
{\left( \dfrac{1}{2}\left[1+ \prod_{i=1}^k\prod_{j=1}^l \dfrac{(1+t_i)}{(1-t_i)} \dfrac{(1+y_j)}{(1-y_j)}\right]\right)}^{j} HS_{\lambda}(\mathrm{T}_k,\mathrm{Y}_l).
\]}
Recall that we can identify $\lambda\in H(k,l)$ with the partitions $\lambda_0,\mu,\nu$ (see Figure \ref{defuv}), so
\[
G^j(\mathrm{V}_k^{\lambda_0}\mathrm{T}_k^{\mu}\mathrm{Y}_l^{\nu})= G^j (M(HS_{\lambda}(\mathrm{T}_k,\mathrm{Y}_l)))=M\left({\left( \dfrac{1}{2}\left[1+ \prod_{i=1}^k\prod_{j=1}^l \dfrac{(1+t_i)}{(1-t_i)} \dfrac{(1+y_j)}{(1-y_j)}\right]\right)}^{j}\cdot HS_{\lambda}(\mathrm{T}_k,\mathrm{Y}_l)\right).
\]
Hence
\[
\widehat{M}(UT_n(E);V_k,T_k,Y_l)=\sum_{j=1}^n\sum_{q=0}^{j-1}\sum_{\lambda\vdash q}(-1)^{j-q-1}{n\choose j}{j-1\choose q}d_{\lambda}G^j(V_k^{\lambda_0}T_k^\mu U_l^\nu),
\]
and the proof follows.
\end{proof}

\section{Applications of the $(k,l)$-multiplicity series of $UT_n(E)$}
Theorems \ref{G1des}, \ref{G2des} and \ref{hookmultUTE} give us an algorithm to calculate the multiplicities in the cocharacter sequence of $UT_n(E)$.
In order to show how the algorithm works, we work on some particular cases.

\begin{theorem}
The $(1,1)$-multiplicity series of $E$ is
\[
\widehat{M}(E;v,t,y)= 1+\frac{v}{(1-t)(1-y)}.
\]
\end{theorem}

\begin{proof}
Using Theorem \ref{hookmultUTE}, we have $\widehat{M}(E;v,t,y)=G(1)$. 
By theorem \ref{G1geral} we get
\[
G_1(1)=G_1(\widehat{M}(HS_{\lambda_0});v,t,y)= Y(1;\mathrm{T}_0)\overline{Y}(1;\mathrm{Y}_0)+ vY(1;\mathrm{T}_1)\overline{Y}(1;\mathrm{Y}_0)= 1+ \dfrac{v}{1-t}.
\]
Now, we are going to compute $G_2(1)$ and $G_2\left(\dfrac{v}{1-t}\right)$. Using 
Corollary \ref{G2des}, we get
\[
G_2(1)= G_2(\widehat{M}(HS_{(0)});v,t,y)= \overline{Y}(1,\mathrm{T}_0)Y(1,\mathrm{Y}_0)+ v\overline{Y}(1,\mathrm{T}_0)Y(1,\mathrm{Y}_1)= 1+\dfrac{v}{1-y}
\]
\[
G_2\left(\dfrac{v}{1-t}\right)= v\overline{Y}\left(\dfrac{1}{1-t};\mathrm{T}_1\right)Y(1;\mathrm{Y}_1)= \dfrac{v}{(1-t)(1-y)}+ \dfrac{vt}{(1-t)(1-y)}.
\]
Hence
\[
G(1)  =  \dfrac{1}{2}\left(\mathbb{1}+ G_2\circ G_1(1)\right)
 =  \dfrac{1}{2}\left(1+ 1+\dfrac{2v}{(1-t)(1-y)}\right)
 =  1 + \dfrac{v}{(1-t)(1-y)},
\]
as desired. 
\end{proof}
Notice that at light of Theorem \ref{ganchoE}, the $(1,1)$-multiplicity series of $E$ contains all the information about the multiplicities in the sequence of cocharacters  of $E$. In fact:
\[
\dfrac{v}{(1-t)(1-y)}=\sum\limits_{m,n\geq 0}vt^ny^m,
\]
then, developing its homogeneous component of degree $m$, we obtain
\[
\sum_{q=1}^{m}vt^{q-1}y^{m-q}
\]
which gives the exact multiplicities of the $m$-th cocharacter of $E$ as in Theorem  \ref{teo KRG}.

By Theorem \ref{ganchoUE}, we have  $\chi(UT_2(E))\subseteq H(2,3)$. Hence we would like to compute the $(2,3)$-multiplicity series of $UT_2(E)$. We need the following  technical lemma.

\begin{lemma}\label{lemaG}
Consider the hook $H(2,3)$ and the set of variables $\{v_1, v_2,t_1,t_2,y_1,y_2,y_3\}$. Then:

\begin{eqnarray*}
 \text{\rm (i) }   G(1,\mathrm{V}_2,\mathrm{T}_2,\mathrm{Y}_3)= 1+ v_1+v_1^2+\dfrac{v_1v_2}{1-y_1}+ \dfrac{v_1^2v_2}{1-y_1} + \dfrac{v_1^3}{1-t_1}+ \dfrac{v_1^3v_2}{(1-t_1)(1-y_1)}; \qquad \qquad \qquad
\end{eqnarray*}
\begin{eqnarray*}
\text{\rm (ii) } 
G^2(1) & = & 1+ 2v_1+ 3v_1^2+ v_1v_2\left(\dfrac{2}{1-y_1} + \dfrac{1}{(1-y_1)^2}\right)\\
& &+ v_1^2v_2\left(\dfrac{4}{1-y_1}+ \dfrac{1+y_1}{(1-y_1)^2} +\dfrac{1}{(1-y_1)^2} \right) \\
& & + v_1^3\left( \dfrac{3}{1-t_1}+ \dfrac{1}{(1-t_1)^2}\right) +  v_1^3v_2\left(\dfrac{5}{(1-t_1)(1-y_1)}\right.\\
& & \left. + \dfrac{2(1+y_1)}{(1-t_1)(1-y_1)^2}+ \dfrac{1+t_1}{(1-t_1)^2(1-y_1)}+
\dfrac{1+t_1y_1}{(1-t_1)^2(1-y_1)^2}\right) \\
& & +v_1^2v_2^2\left(\dfrac{2}{1-y_1}+ \dfrac{2(1+y_1)}{(1-y_1)^2} + \dfrac{4y_1y_2}{(1-y_1)^2(1-y_1y_2)}\right)\\
& & + v_1^3v_2^2\left(\dfrac{3}{(1-t_1)(1-y_1)}+\dfrac{3(1+y_1)}{(1-t_1)(1-y_1)^2}\right. \\
& & + \dfrac{6y_1y_2}{(1-t_1)(1-y_1)^2(1-y_1y_2)}+ \dfrac{1+t_1}{(1-t_1)^2(1-y_1)}  \\
&  & +\left.   \dfrac{(1+t_1)(1+y_1)}{(1-t_1)^2(1-y_1)^2} +\dfrac{2(1+t_1)y_1y_2}{(1-t_1)^2(1-y_1)^2(1-y_1y_2)}\right)\\
& & + v_1^3v_2^3\left(\dfrac{1}{(1-t_1)(1-y_1)}+\dfrac{1+y_1}{(1-y_1)^2(1+t_1)} \right. \\
& & + \dfrac{2y_1y_2}{(1-t_1)(1-y_1)^2(1-y_1y_2)}+ \dfrac{1+t_1+t_1t_2-t_1^2t_2}{(1-t_1)^2(1-t_1t_2)(1-y_1)} \\
& & +\left.\dfrac{(1+t_1+t_1t_2-t_1^2t_2)(1+y_1)}{(1-t_1)^2(1-t_1t_2)(1-y_1)^2} + \dfrac{2(1+t_1+t_1t_2-t_1^2t_2)y_1y_2}{(1-t_1)^2(1-t_1t_2)(1-y_1)^2(1-y_1y_2)}\right);
\end{eqnarray*}
\begin{eqnarray*}
\text{\rm (iii) }
 G^2(v_1)& = & v_1+2v_1^2 +v_1v_2\left(\dfrac{1}{1-y_1}+ \dfrac{1}{(1-y_1)^2}\right)\\
& &+v_1^2v_2\left(\dfrac{3}{1-y_1}+\dfrac{1+y_1}{(1-y_1)^2}+\dfrac{2}{(1-y_1)^2}\right) \\
& & + v_1^3v_2\left(\dfrac{4}{(1-t_1)(1-y_1)}+ \dfrac{3(1+y_1)}{(1-t_1)(1-y_1)^2}+ \dfrac{2(1+t_1y_1)}{(1-t_1)^2(1-y_1)^2} \right.\\
& &+\left. \dfrac{1+t_1}{(1-t_1)^2(1-y_1)}\right) + v_1^3\left(\dfrac{2}{1-t_1}+\dfrac{1}{(1-t_1^2)}\right)\\
& & + v_1^2v_2^2\left(\dfrac{2}{1-y_1}+\dfrac{3(1+y_1)}{(1-y_1)^2}+\dfrac{6y_1y_2}{(1-y_1)^2(1-y_1y_2)}+ \dfrac{(1+y_1y_2)}{(1-y_1)^2(1-y_1y_2)}\right)\\
& &  + v_1^3v_2^2\left( \dfrac{3}{(1-t_1)(1-y_1)}+\dfrac{5(1+y_1)}{(1-t_1)(1-y_1)^2}+ \dfrac{10y_1y_2}{(1-t_1)(1-y_1)^2(1-y_1y_2)}\right.\\
& &+ \dfrac{2(1+t_1)(1+y_1)}{(1-t_1)^2(1-y_1)^2}+\dfrac{4(1+t_1)y_1y_2}{(1-t_1)^2(1-y_1)^2(1-y_1y_2)}\\
& & \left. +\dfrac{1+y_1y_2+2t_1y_1y_2+t_1y_1-t_1y_1^2y_2}{(1-t_1)^2(1-y_1)^2(1-y_1y_2)} +\dfrac{1+3y_1y_2-y_1^2y_2+y_1}{(1-t_1)(1-y_1)^2(1-y_1y_2)}\right)\\\\
& & +v_1^3v_2^3\left( \dfrac{1}{(1-t_1)(1-y_1)}+ \dfrac{2(1+y_1)}{(1-t_1)(1-y_1)^2}\right.\\
& &+\dfrac{4y_1y_2}{(1-t_1)(1-y_1)^2(1-y_1y_2)}+\dfrac{(1+t_1+t_1t_2-t_1^2t_2)}{(1-t_1)^2(1-t_1t_2)(1-y_1)}\\
& &+\dfrac{2(1+t_1+t_1t_2-t_1^2t_2)(1+y_1)}{(1-t_1)^2(1-t_1t_2)(1-y_1)^2}+ \dfrac{4(1+t_1+t_1t_2-t_1^2t_2)y_1y_2}{(1-t_1)^2(1-t_1t_2)(1-y_1)^2(1-y_1y_2)} \\
& &+\dfrac{(1+t_1+t_1t_2-t_1^2t_2)(1-y_1^2y_2+3y_1y_2+y_1+2y_1y_2y_3)}{(1-t_1)^2(1-t_1t_2)(1-y_1)^2(1-y_1y_2)} 
\end{eqnarray*}
\begin{eqnarray*}
 +\left. \dfrac{(1-y_1^2y_2+3y_1y_2+y_1+2y_1y_2y_3)}{(1-t_1)(1-y_1)^2(1-y_1y_2)} \right). \qquad\qquad\qquad\qquad\qquad
\end{eqnarray*}

\end{lemma}

\begin{proof}
(i) The result follows directly from the definition of $G$.

(ii) By the previous item and the linearity of $G$, we have
\begin{equation}\label{G21}
\begin{split}
G^2(1)= & G(1)+ G(v_1)+G(v_1^2)+G\left(\dfrac{v_1v_2}{1-y_1}\right)+ G\left(\dfrac{v_1^2v_2}{1-y_1}\right) \\
&+ G\left(\dfrac{v_1^3}{1-t_1}\right)+ G\left(\dfrac{v_1^3v_2}{(1-t_1)(1-y_1)}\right).
\end{split}
\end{equation}
We are going to compute each part of the right hand side of the equality above separately.\\
\indent
Let us calculate $G(v_1)$. We start computing $G_1(v_1)$. From Theorem \ref{G1des}, we get
\begin{equation*}
\begin{split}
G_1(v_1)= & v_1 +v_1^2+v_1v_2+v_1^2v_2+ \dfrac{v_1^3v_2}{1-t_1}+ \dfrac{v_1^3}{1-t_1}.
\end{split}
\end{equation*}
It follows that
\begin{equation*}
\begin{split}
G_2(G_1(v_1))= & G_2(v_1) +G_2(v_1^2)+G(v_1v_2)+G_2(v_1^2v_2)+ G_2\left(\dfrac{v_1^3v_2}{1-t_1}\right)+ G_2\left(\dfrac{v_1^3}{1-t_1}\right).
\end{split}
\end{equation*}
By Corollary \ref{G2des}, we have
\begin{equation*}
\begin{split}
G_2(G_1(v_1))= & v_1 + 2v_1^2 + \dfrac{2v_1v_2}{1-y_1}+ \dfrac{2v_1^3}{1-t_1}+ \dfrac{4v_1^2v_2}{1-y_1}+ \dfrac{4v_1^3v_2}{(1-t_1)(1-y_1)}\\
&  +\dfrac{2v_1^2v_2^2}{1-y_1}+
  \dfrac{2v_1^3v_2^2}{(1-t_1)(1-y_1)}.
\end{split}
\end{equation*}
Hence
\begin{equation}\label{Gv1}
\begin{split}
G(v_1)= & v_1 + v_1^2 + \dfrac{v_1v_2}{1-y_1}+ \dfrac{2v_1^2v_2}{1-y_1}+\dfrac{v_1^3}{1-t_1}\\
& +\dfrac{2v_1^3v_2}{(1-t_1)(1-y_1)}+ \dfrac{v_1^2v_2^2}{1-y_1}+
  \dfrac{v_1^3v_2^2}{(1-t_1)(1-y_1)}.
\end{split}
\end{equation}\\

To calculate the remaining parts, we shall strongly use Theorem \ref{G1des} and Corollary \ref{G2des}.
As the computations are too many, we shall only write the final results, as we can see below.
\begin{equation}
\begin{split}
G(v_1^2)=& v_1^2 + \dfrac{v_1^2v_2}{1-y_1}+\dfrac{v_1^3}{1-t_1}+
\dfrac{2v_1^3v_2}{(1-t_1)(1-y_1)}+ \dfrac{v_1^2v_2^2}{1-y_1}\\
  & +  \dfrac{2v_1^3v_2^2}{(1-t_1)(1-y_1)}+\dfrac{v_1^3v_2^3}{(1-t_1)(1-y_1)};
\end{split}
\end{equation}
\vspace{0.2cm}
\begin{equation*}
\begin{split}
G\left(\dfrac{v_1v_2}{1-y_1}\right)= &  \dfrac{v_1v_2}{(1-y_1)^2}+ v_1^2v_2\left(\dfrac{1+y_1}{(1-y_1)^2}\right)+ v_1^2v_2^2\left(\dfrac{1+y_1}{(1-y_1)^2}+ \dfrac{2y_1y_2}{(1-y_1)^2(1-y_1y_2)}\right)\\
 & +v_1^3v_2\left(\dfrac{1+y_1}{(1-t_1)(1-y_1)^2}\right)+
v_1^3v_2^2 \left(\dfrac{1+y_1}{(1-t_1)(1-y_1)^2}+\dfrac{2y_1y_2}{(1-y_1)^2(1-y_1y_2)}\right);
\end{split}
\end{equation*}
\vspace{0.2cm}
\begin{equation*}
\begin{split}
G\left(\dfrac{v_1^3}{1-t_1}\right)= &  \dfrac{v_1^3}{(1-t_1)^2}+
  v_1^3v_2\left(\dfrac{1+t_1}{(1-t_1)^2(1-y_1)}\right)+
v_1^3v_2^2 \left(\dfrac{1+t_1}{(1-t_1)^2(1-y_1)}\right)\\
& +v_1^3v_2^3\left(\dfrac{1+t_1+t_1t_2-t_1^2t_2}{(1-t_1)^2(1-t_1t_2)(1-y_1)}\right);
\end{split}
\end{equation*}
\vspace{0.1 cm}
\begin{equation*}
\begin{split}
G\left(\dfrac{v_1^2v_2}{1-y_1}\right)= &  \dfrac{v_1^2v_2}{(1-y_1)^2}+
  v_1^3v_2\left(\dfrac{1+y_1}{(1-t_1)(1-y_1)^2}\right)\\
 & +v_1^2v_2^2\left(\dfrac{1+y_1}{(1-y_1)^2} +\dfrac{2y_1y_2}{(1-y_1)^2(1-y_1y_2)}\right)\\
 &+ v_1^3v_2^2 \left(\dfrac{2(1+y_1)}{(1-t_1)(1-y_1)^2}  + \dfrac{4y_1y_2}{(1-t_1)(1-y_1)^2(1-y_1y_2)}\right)\\
 &+ v_1^3v_2^3 \left(\dfrac{(1+y_1)}{(1-t_1)(1-y_1)^2}  +\dfrac{2y_1y_2}{(1-t_1)(1-y_1)^2(1-y_1y_2)}\right);
\end{split}
\end{equation*}
\begin{equation*}
\begin{split}
G\left(\dfrac{v_1^3v_2}{(1-t_1)(1-y_1)}\right)= &
  v_1^3v_2\left(\dfrac{1+t_1y_1}{(1-t_1)^2(1-y_1)^2}\right)\\
&+ v_1^3v_2^2 \left(\dfrac{(1+t_1)(1+y_1)}{(1-t_1)^2(1-y_1)^2}
+ \dfrac{2(1+t_1)y_1y_2}{(1-t_1)^2(1-y_1)^2(1-y_1y_2)}\right)\\
& +v_1^3v_2^3 \left(\dfrac{(1+t_1+t_1t_2-t_1^2t_2)(1+y_1)}{(1-t_1)^2(1-t_1t_2)(1-y_1)^2}
+\dfrac{2(1+t_1+t_1t_2-t_1^2t_2)y_1y_2}{(1-t_1)^2(1-t_1t_2)(1-y_1)^2(1-y_1y_2)}\right).
\end{split}
\end{equation*}
By the equality (\ref{G21}) and the previous computations, we obtain the desired the result.

(iii) From the equality (\ref{Gv1}), we have
\begin{equation}\label{G2v1}
\begin{split}
G^2(v_1)= & G(v_1) + G(v_1^2) + G\left(\dfrac{v_1v_2}{1-y_1}\right)+ 2G\left(\dfrac{v_1^2v_2}{1-y_1}\right)+G\left(\dfrac{v_1^3}{1-t_1}\right)\\
& +2G\left(\dfrac{v_1^3v_2}{(1-t_1)(1-y_1)}\right)+ G\left(\dfrac{v_1^2v_2^2}{1-y_1}\right)+ G\left(  \dfrac{v_1^3v_2^2}{(1-t_1)(1-y_1)}\right).
\end{split}
\end{equation}
It only remains to calculate $G\left(\dfrac{v_1^2v_2^2}{1-y_1}\right)$ and  $G\left(\dfrac{v_1^3v_2^2}{(1-t_1)(1-y_1)}\right)$, since the other summands have been calculated in item (ii).

The following equalities are obtained by Theorem \ref{G1des} and Corollary \ref{G2des}.
\begin{equation*}
\begin{split}
G\left(\dfrac{v_1^2v_2^2}{1-y_1}\right)= & v_1^2v_2^2\left(\dfrac{1+y_1y_2}{(1-y_1)^2(1-y_1y_2)}\right)+ v_1^3v_2^2\left(\dfrac{1+3y_1y_2-y_1^2y_2+y_1}{(1-y_1)^2(1-y_1y_2)(1-t_1)}\right)\\
& +v_1^3v_2^3\left(\dfrac{1+3y_1y_2-y_1^2y_2+y_1+2y_1y_2y_3}{(1-y_1)^2(1-y_1y_2)(1-t_1)}\right);
\end{split}
\end{equation*}
\begin{equation*}
\begin{split}
G\left(\dfrac{v_1^3v_2^2}{(1-t_1)(1-y_1)}\right)= & v_1^3v_2^2\left(\dfrac{1+y_1y_2+2t_1y_1y_2+t_1y_1-t_1y_1^2y_2}{(1-y_1)^2(1-y_1y_2)(1-t_1)^2}\right)\\
&+ v_1^3v_2^3\left(\dfrac{(1+t_1+t_1t_2-t_1^2t_2)(1+3y_1y_2-y_1^2y_2+y_1+2y_1y_2y_3)}{(1-t_1)^2(1-t_1t_2)(1-y_1)^2(1-y_1y_2)}\right).
\end{split}
\end{equation*}
Now the result follows by a combination of the equality (\ref{G2v1}) and the previous computations.
\end{proof}

\begin{theorem}\label{SMUT2}
The $(2,3)$-multiplicity series of $UT_2(E)$ is
\begin{equation*}
\begin{split}
\widehat{M}(UT_2(E);\mathrm{V}_2, \mathrm{T}_2,\mathrm{Y}_3) = &  1+v_1+v_1^2+\dfrac{v_1v_2}{1-y_1}+ \dfrac{v_1^3}{1-t_1}+\dfrac{ v_1^2v_2(2-y_1)}{(1-y_1)^2}\\
 & +v_1^3v_2 \left(\dfrac{1}{(1-t_1)(1-y_1)}+ \dfrac{1+y_1}{(1-t_1)(1-y_1)^2}+\dfrac{1+t_1y_1}{(1-t_1)^2(1-y_1)^2}\right)\\
 & +v_1^2v_2^2\left(\dfrac{2+y_1}{(1-y_1)^2}+ \dfrac{4y_1y_2}{(1-y_1)^2(1-y_1y_2)}\right)\\
 &+v_1^3v_2^2\left(  \dfrac{2(1+y_1)}{(1-t_1)(1-y_1)^2}+ \dfrac{4y_1y_2}{(1-t_1)(1-y_1)^2(1-y_1y_2)}\right.\\
 & + \dfrac{(1+t_1)(1+y_1)}{(1-t_1)^2(1-y_1)^2} +  \dfrac{2(1+t_1)y_1y_2}{(1-t_1)^2(1-y_1)^2(1-y_1y_2)}\\
 &+\left. \dfrac{1+y_1y_2+2t_1y_1y_2+t_1y_1-t_1y_1^2y_2}{(1-t_1)^2(1-y_1)^2(1-y_1y_2)} + \dfrac{1+3y_1y_2-y_1^2y_2+y_1}{(1-t_1)(1-y_1)^2(1-y_1y_2)}\right)\\
 &+v_1^3v_2^3\left(\dfrac{(1+y_1)}{(1-t_1)(1-y_1)^2} +  \dfrac{2y_1y_2}{(1-t_1)(1-y_1)^2(1-y_1y_2)}\right.\\
&+\dfrac{(1+t_1+t_1t_2-t_1^2t_2)(1+y_1)}{(1-t_1)^2(1-t_1t_2)(1-y_1)^2}+\dfrac{2(1+t_1+t_1t_2-t_1^2t_2)y_1y_2}{(1-t_1)^2(1-t_1t_2)(1-y_1)^2(1-y_1y_2)} \\
 & +\dfrac{(1+t_1+t_1t_2-t_1^2t_2)(1-y_1^2y_2+3y_1y_2+y_1+2y_1y_2y_3)}{(1-t_1)^2(1-t_1t_2)(1-y_1)^2(1-y_1y_2)} \\
& +\left. \dfrac{(1-y_1^2y_2+3y_1y_2+y_1+2y_1y_2y_3)}{(1-t_1)(1-y_1)^2(1-y_1y_2)} \right).
\end{split}
\end{equation*}
\end{theorem}

\begin{proof}
Proposition \ref{ganchoUE} implies that $\chi(UT_2(E))\subseteq H(2,3)$. Hence we can work with the set of variables $\{v_1, v_2,t_1,t_2,y_1,y_2,y_3\}$.
By Theorem \ref{hookmultUTE} we obtain
\begin{equation}\label{MUT2}
\widehat{M}(UT_2(E);\mathrm{V}_2, \mathrm{T}_2,\mathrm{Y}_3)=2G(1)-G^2(1)+G^2(v_1).
\end{equation}
Now the result follows from the equality (\ref{MUT2}) and Lemma \ref{lemaG}.
\end{proof}

The next result was proved by Centrone in \cite{C}. Now, we are going to prove it using the $(2,3)$-multiplicity series of $UT_2(E)$.

\begin{corollary}
Let $\lambda$ be a partition. The multiplicity $m_{\lambda}$ in the cocharacter sequence of $UT_2(E)$ is given by the following expressions:

{\small
$m_{\lambda}= \left\{ \begin{array}{lcl}
             1 &   if  & \lambda=(n),  \\
             1 & if & \lambda=(1^m),\ m>1,\\
             m+1 & if & \lambda=(2,1^m),\ m\geq 1,\\
             3m+2 & if & \lambda=(2,2,1^m),\ m\geq 0,\\
             4(m+1) & if & \lambda=(2,2,2^s,1^m),\ m\geq 0,\ s>0,\\
             2nm-3m-n+3 & if & \lambda=(n,1^m),\ n\geq 3,\ m\geq 1,\\
             6m(n-3)+9m+3(n-3)+5 & if & \lambda=(n,2,1^m),\ n\geq3, \ m\geq 0,\\
             (8(n-3)+12)(m+1) & if & \lambda=(n,2,2^s,1^m),\ n\geq 3,\ s\geq 1,\ m\geq 0,\\
             4(n_1-n_2+1)(2m+1) & if & \lambda=(n_1,n_2,1^m),\ n_1\geq n_2\geq 3,\ m\geq 0,\\
             12(n_1-n_2+ 1)(m+1) & if & \lambda=(n_1,n_2,2^s,1^m),\ n_1\geq n_2\geq 3,\ s\geq 1,\ m\geq 0,\\
             4(n_1-n_2+1)(m+1) & if & \lambda=(n_1,n_2,3,2^s,1^m),\ n_1\geq n_2\geq 3,\ s\geq 0,\ m\geq 0,\\
             0 & \textrm{for}& \textrm{all other}  \ \lambda.
             \end{array}
   \right.$
}
\end{corollary}

\begin{proof}
Given a partition $\lambda$, by Theorem \ref{ganchoUE} we know that $m_{\lambda}=0$ if $\lambda\notin H(2,3)$. Hence let $\lambda\in H(2,3)$. In order to compute
the multiplicity $m_{\lambda}$, it is necessary to write the hook multiplicity  series of $UT_2(E)$ as a power series.

Let $\lambda\in H(2,3)$ and consider the triple $(\lambda_0,\mu,\nu)$. Notice that $D_{\lambda_0}\subseteq D_{(3,3)}$. It follows that
\[
\lambda_0\in \{(1),(2),(3),(1,1),(2,1),(3,1),(2,2),(3,2),(3,3)\}.
\]
First, let $\lambda$ be a partition such that $\lambda_0\in\{(1),(2),(3)\}$. Using Theorem \ref{SMUT2}, we obtain $m_{\lambda}=1$.

Consider now $\lambda$ such that $\lambda_0=(1,1)$. By Theorem \ref{SMUT2}, we have that $\lambda_0$  corresponds to the summand
\[
\dfrac{v_1v_2}{1-y_1}=v_1v_2\sum\limits_{m\geq 0} y_1^m.
\]
It follows that the partitions of type $\lambda=(1,1,1^m)$ with $m\geq 0$ have multiplicity $1$ or, equivalently, if $\lambda=(1^m)$, with $m> 1$ then $m_{\lambda}=1$.

Now, let $\lambda$ be such that $\lambda_0=(2,1)$. Observe that
\begin{equation*}
\dfrac{ v_1^2v_2(2-y_1)}{(1-y_1)^2}=v_1^2v_2\left(2 + \sum\limits_{n\geq 1}(m+2)y_1^m\right).
\end{equation*}
Hence, if $\lambda=(2,1)$, then $m_{\lambda}=2$. Moreover, if $\lambda=(2,1,1^m)$, with $m\geq 1$, then $m_{\lambda}=m+2$ or, equivalently, if $\lambda=(2,1^m)$, with $m\geq 2$, then $m_{\lambda}=m+1$.

Let $\lambda$ be a partition such that $\lambda_0=(2,2)$. Then
\begin{equation*}
 v_1^2v_2^2\left(\dfrac{2+y_1}{(1-y_1)^2}+ \dfrac{4y_1y_2}{(1-y_1)^2(1-y_1y_2)}\right)=v_1^2v^2\left(2+ \sum\limits_{m\geq 1} (3m+2)y_1^m+ \sum\limits_{m\geq 1}\sum\limits_{s\geq 1}4my_1^{m+s-1}y_2^s\right).
\end{equation*}
Hence if $\lambda=(2,2)$, then $m_{\lambda}=2$. If $\lambda=(2,2,1^m)$, then $m_{\lambda}=3m+2$.

Observe that $y_1^{m+s-1}y_2^s$ is in one-to-one correspondence with the partition $\nu=(m+s-1,s)$, hence $\nu'=(2^s,1^{m-1})$.
So, if $\lambda=(2,2,2^s, 1^{m-1})$, with $m,s\geq 1$, then $m_{\lambda}=4m$ or, equivalently, if $\lambda=(2,2,2^s,1^m)$, with $s\geq 1$ and $m\geq 0$, then $m_{\lambda}=4(m+1)$.

The other cases are treated similarly.
\end{proof}
Now, we are going to calculate the multiplicities $m_{\lambda}$ in the cocharacter sequence of $UT_3(E)$ when $\lambda\in H(1,1)$.

\begin{theorem}
Let $\lambda$ be a partition such that
$\lambda\in H(1,1)$. The multiplicity $m_{\lambda}$ in the cocharacter sequence of $UT_3(E)$ is given by the following expressions:

{\small $m_{\lambda}= \left\{ \begin{array}{lcl}
             1 &   if  & \lambda=(n),\ n\geq 0,  \\
             1 & if & \lambda=(1^m),\ m>1,\\
             n & if & \lambda=(n,1), \ n\geq 2,\\
             m+1 & if & \lambda=(2,1^m),\ m\geq 2,\\
           \dfrac{1}{4}(76-90m+26m^2-54n+ 68mn-20m^2n & \\
           + 10n^2-12mn^2+4m^2n^2)&if &\lambda=(n,1^m),\ n\geq 3,m\geq 2.
             \end{array}
   \right.$
}
\end{theorem}

\begin{proof}
By Theorem \ref{hookmultUTE}, we have
\begin{equation}\label{h11ut3}
\widehat{M}(UT_3(E);t,y,v)= 3G(1)-3G^2(1)+3G^2(v)+ G^3(1)-2G^3(v)+G^3(vt)+ G^3(vy)
\end{equation}
By Theorem \ref{G1des} and Corollary \ref{G1des}, we obtain
\begin{itemize}
\item $G(1)=1+\dfrac{v}{(1-t)(1-y)}$;
\item $G^2(1)= 1+\dfrac{v}{(1-t)(1-y)}+ \dfrac{v(1+ty)}{(1-t)^2(1-y)^2}$;
\item $G^2(v)= \dfrac{v(1+2ty+t^2y^2)}{(1-t)^2(1-y)^2}$;
\item $G^3(1)=1+\dfrac{v}{(1-t)(1-y)}+ \dfrac{v(1+ty)}{(1-t)^2(1-y)^2}+  \dfrac{v(1+2ty+t^2y^2)}{(1-t)^3(1-y)^3}$;
\item $G^3(v)= \dfrac{v(1+3ty+3t^2y^2+ t^3y^3)}{(1-t)^3(1-y)^3}$;
\item $G^3(vt)= \dfrac{v(t+ 3t^2y+ 3t^3y^2+ t^4y^3)}{(1-t)^3(1-y)^3}$;
\item $G^3(vy)= \dfrac{v(y+3ty^2+ 3t^2y^2+ t^3y^4)}{(1-t)^3(1-y)^3}$.
\end{itemize}
By the equation (\ref{h11ut3}) we obtain that the $(1,1)$-multiplicity series of $UT_3(E)$ in the variables $v,t,y$ is
\begin{equation}\label{m11ut3}
\begin{split}
\widehat{M}(UT_3(E);v,t,y) = & 1 + \dfrac{v}{(1-t)(1-y)}+\dfrac{v(1+4ty+ 3t^2y^2)}{(1-t)^2(1-y)^2}\\
&+ \dfrac{v}{(1-t)^3(1-y)^3}\left(-1+t+y-4ty-5t^2y^2-2t^3y^3
  +3t^2y + 3t^3y^2+ t^4y^3+ 3ty^2+ 3t^2y^3+ t^3y^4\right).
\end{split}
\end{equation}
Note that to calculate the multiplicity $m_{\lambda}$ where $\lambda\in H(1,1)$, it is necessary to write (\ref{m11ut3}) as a power series. Recall that
\begin{equation*}
\dfrac{t^{a_1}y^{a_2}}{(1-t)(1-y)}= \sum\limits_{n\geq a_1}\sum\limits_{m\geq a_2}t^ny^m,
\end{equation*}
\begin{equation*}
\dfrac{t^{a_1}y^{a_2}}{(1-t)^2(1-y)^2}= \sum\limits_{n\geq a_1}\sum\limits_{m\geq a_2}(n-a_1+1)(m-a_2+1)t ^{n}y^{m},
\end{equation*}
\begin{equation*}
\dfrac{t^{a_1}y^{a_2}}{(1-t)^3(1-y)^3}=\sum\limits_{n_\geq {a_1}}\sum\limits_{m\geq a_2} {n-a_1+2\choose 2}{m-a_2+2\choose 2}t^{n}y^{m}.
\end{equation*}
Using the previous equations and making some algebraic manipulations, we obtain the following  expression
{\small
\begin{equation}\label{multut311}
\begin{split}
\widehat{M}(UT_3(E);v,t,y)= & 1+ v\left(\sum\limits_{n\geq 0}t^n+ \sum\limits_{m\geq 1} y^m+ \sum\limits_{n\geq 1}(n+1)t^ny + \sum\limits_{m\geq 2}(m+1)ty^m \right.\\
& +\left. \sum\limits_{n\geq 2}\sum\limits_{m\geq 2}\dfrac{(32-34(m+n)+10(n^2+m^2)-12(m^2n+n^2m)+ 44mn+ 4m^2n^2)}{4}t^ny^m\right).
\end{split}
\end{equation}
}
By the equation (\ref{multut311}), it follows that if $\lambda=(n)$ or $\lambda=(1^m)$ then $m_{\lambda}=1$.
Now, if $\lambda=(n+1,1)$ and $n\geq 1$ then $m_{\lambda}= n+1$, which means that if $\lambda=(n,1)$ with $n\geq 2$ then $m_{\lambda}=n$.
Observe that if $\lambda=(2,1^m)$ with $m\geq 2$, its multiplicity is $m+1$.\\
Finally if $\lambda=(n+1,1^m)$ with $n,m\geq 2$, then we have that
\[
m_{\lambda}=\dfrac{32-34(m+n)+10(n^2+m^2)-12(mn^2+n^2m)+ 44mn+ 4m^2n^2}{4},
\]
or equivalently if $\lambda=(n, 1^{m})$ with $n\geq 3$ and $m\geq 2$, we have that
\[
m_{\lambda}=\dfrac{76-90m+26m^2-54n+ 68mn-20m^2n+ 10n^2-12mn^2+4m^2n^2}{4}.
\]
\end{proof}

Recall that if we want to know all multiplicities of cocharacter sequences of $UT_3(E)$,
we have to work with the hook $H(3,5)$ because by Theorem \ref{ganchoUE} we know that $\chi (UT_3(E))\subseteq H(3,5)$.
Hence the $(3,5)$-multiplicity series $\widehat{M}(UT_3(E);\mathrm{V}_3,\mathrm{T}_3,\mathrm{Y}_5)$ has 11 variables and the computations are very technical.

\end{document}